\documentclass[11pt,a4paper]{article}
\usepackage[T1]{fontenc}
\usepackage[utf8]{inputenc}
\usepackage{lmodern}
\usepackage[margin=27mm,headheight=14pt]{geometry}
\usepackage{amsmath,amssymb,amsthm,mathtools}
\usepackage{microtype,enumitem,booktabs,array}
\usepackage[numbers,sort&compress]{natbib}
\usepackage[hidelinks,bookmarksnumbered]{hyperref}
\usepackage{fancyhdr}
\allowdisplaybreaks[2]
\numberwithin{equation}{section}
\newtheorem{theorem}{Theorem}[section]
\newtheorem{proposition}[theorem]{Proposition}
\newtheorem{lemma}[theorem]{Lemma}
\newtheorem{corollary}[theorem]{Corollary}
\theoremstyle{definition}
\newtheorem{definition}[theorem]{Definition}
\newtheorem{remark}[theorem]{Remark}
\setlist[enumerate]{label=\textup{(\roman*)},leftmargin=2.4em}
\newcommand{\R}{\mathbb R}
\newcommand{\C}{\mathbb C}
\newcommand{\dd}{\,\mathrm d}
\newcommand{\supp}{\operatorname{supp}}
\newcommand{\Rea}{\operatorname{Re}}
\newcommand{\Ima}{\operatorname{Im}}
\newcommand{\diver}{\operatorname{div}}
\newcommand{\ind}{\mathbf1}
\newcommand{\wlim}{\mathop{\mathrm{w\!-\!lim}}}
\newcommand{\crit}[2]{\left\|#1\right\|_{L_t^{\frac{2(d+4-2b)}{d}}L_x^{\frac{2(d+4-2b)}{d+2-2b},2}(#2\times\R^d)}}
\newcommand{\dual}[2]{\left\|#1\right\|_{L_t^2L_x^{\frac{2d}{d+2},2}(#2\times\R^d)}}

\makeatletter
\renewcommand{\@biblabel}[1]{#1.}
\makeatother
\hypersetup{pdftitle={Classification of threshold dynamics for the focusing mass-critical inhomogeneous NLS},pdfauthor={},pdfsubject={Focusing mass-critical inhomogeneous NLS}}
\title{Classification of threshold dynamics for the focusing mass-critical inhomogeneous NLS}
\author{Xuan Liu}
\date{}

\begin{document}
	\maketitle

\begin{abstract}
	We study the focusing mass-critical inhomogeneous nonlinear Schr\"odinger equation
	$i\partial_{t}u+\Delta u+|x|^{-b}|u|^{\frac{4-2b}{d}}u=0$
	in dimensions $d\ge3$.  For $\frac12<b<\min\{2,d/2\}$, we prove that
	every forward-global $H^1$ solution with mass
	equal to that of the ground state $Q$ either
	scatters forward or coincides with $e^{it}Q$,
	up to phase rotation and scaling.  As a consequence, every solution with nonradial $H^1$ initial data at the ground-state mass must be global and scatter in both time directions. For $0<b\le\frac12$, analogous rigidity and scattering results hold under an additional finite-variance assumption. The key new ingredients in our proof are a nonradial incoming/outgoing decomposition and exterior frequency-decay estimates, which yield the kinetic-energy localization needed for a localized virial argument.
\end{abstract}

	
\section{Introduction}\label{sec:intro}
\newtheorem{defn}{Definition}[section]
We consider the focusing mass-critical inhomogeneous nonlinear Schr\"odinger equation

\begin{equation}\label{eq:inls}
	\begin{cases}
		i\partial_tu+\Delta u+|x|^{-b}|u|^{\frac{4-2b}{d}}u=0,\\
		u(0,x)=u_0(x)\in H^1(\R^d).
	\end{cases}
\end{equation}
where $d\ge3$ and $0<b<\min\{2,d/2\}$. The coefficient $|x|^{-b}$ describes a spatially inhomogeneous nonlinear response, a setting motivated in part by models from nonlinear optics; see \cite{Gill00,LT94}. Such power-law coefficients also arise naturally in the study of nonlinearities that decay at spatial infinity; see \cite{GS08}.  

Our aim is to understand the dynamics of \eqref{eq:inls} when the mass of the initial data equals that of the ground state. Before discussing the threshold problem, we recall the notions of strong solutions, blowup, and scattering that will be used throughout the paper.

\begin{defn}[Strong solution]\label{def:solution}
	A function $u : I\times \mathbb{R}^d \to \mathbb{C}$ is a strong solution of \eqref{eq:inls} if, for every compact interval $J\subset I$, it belongs to $C_t^0L^2_x(J\times \mathbb{R} ^d)\cap L^{\frac{2d+8-4b}{d}} _tL_x^{ \frac{2d+8-4b}{d+2-2b},2}(J\times \mathbb{R} ^d)$ and satisfies the Duhamel formula for all $t, t_0 \in I$:
	\begin{equation}
		u(t) = e^{i(t - t_0)\Delta} u(t_0) +i \int_{t_0}^{t} e^{i(t-\tau)\Delta} [ |x|^{-b}|u(\tau)|^ \frac{4-2b}{d} u(\tau) ]d\tau.\notag
	\end{equation}
	 The interval $I$ is called the \textit{lifespan} of $u$. The solution is \textit{maximal-lifespan} if it cannot be extended to a strictly larger interval, and \textit{global} if $I = \mathbb{R}$.
\end{defn}

Associated with the strong-solution class is a notion of blowup expressed in terms of the critical spacetime norm.

\begin{defn}[Blowup]
	A solution $u$ to \eqref{eq:inls} is said to blow up forward in time if there exists $t_0\in I$ such that
	\begin{equation}
		\|u\|_{L_t^{\frac{2d+8-4b}{d}} L_x^{ \frac{2d+8-4b}{d+2-2b},2}([t_0,\sup(I))\times\mathbb{R}^d)}
		=\infty.\notag
	\end{equation}
	Blowup backward in time is defined analogously.
\end{defn}

 By the local theory recalled in Proposition~\ref{prop:local}, a global solution fails to scatter in a given time direction precisely when the critical spacetime norm is infinite in the same direction. Here scattering   means the existence of asymptotic states $u_\pm\in L^2(\mathbb R^d)$ such that
 \begin{equation}
 	\lim_{t\to\pm\infty}\|u(t)-e^{it\Delta}u_\pm\|_{L^2} = 0.\notag
 \end{equation}

The mass is conserved along the flow of \eqref{eq:inls}:

\begin{equation}
	M(u) =: \int_{\mathbb{R}^d}  |u(t,x)|^2\,dx=M(u_0).\label{eq:mass-conservation-intro}
\end{equation}

The equation is called \emph{mass-critical} because the scaling $u(t,x)\mapsto \lambda ^{\frac{d}{2}} u(\lambda^2 t,\lambda x)$ preserves both the equation and its mass. Using Strichartz estimates in Lorentz spaces, Aloui and Tayachi~\cite{AT21} established an $L^2$ local theory for \eqref{eq:inls}.

Before turning to the global dynamics of \eqref{eq:inls}, we briefly recall some results for the classical mass-critical nonlinear Schr\"odinger equation. In the focusing case, the equation takes the form
\begin{equation}
	i\partial_tu+\Delta u+|u|^{4/d}u=0.\notag
\end{equation}
By developing long-time Strichartz estimates, Dodson~\cite{D15} established global well-posedness and scattering in all dimensions $d\ge1$ for initial data with mass strictly below that of the ground state. For corresponding results in the defocusing case, we refer to \cite{TVZ07,KTV09,KVZ08,D12,D16a,D16b}.

At the ground-state mass, the situation is fundamentally different, as the equation admits both solitary waves and finite-time blowup solutions. Merle~\cite{M93} classified the finite-time minimal-mass blowup solutions. For forward-global nonscattering solutions, Killip, Li, Visan and Zhang~\cite{KLVZ09} established a rigidity result for radial $H^1$ data in dimensions $d\ge4$. Subsequently, Li and Zhang~\cite{LZ12} extended this result to radial $H^1$ data in dimensions $d=2,3$ and to certain splitting-spherically symmetric data in dimensions $d\ge4$. More recently, Dodson~\cite{D24,D23} completed the classification in the $L^2$ setting without symmetry assumptions for all dimensions $1\le d\le15$: every $L^2$  solution at the ground-state mass that fails to scatter in at least one time direction is, up to the symmetries of the equation, either a ground-state solitary wave or its pseudoconformal transform.

These results provide a natural motivation for studying the corresponding threshold dynamics of the inhomogeneous equation \eqref{eq:inls}. As in the classical case, the behavior of solutions is closely related to the ground state. Let $Q$ be the unique positive solution of
\begin{equation}
	\Delta Q+|x|^{-b}Q^{\frac{4-2b}{d}+1}=Q. \notag
\end{equation}
Existence results can be found in \cite{G08,GS08} for $d\ge2$ and in \cite{G10} for $d=1$. Uniqueness was established by Yanagida~\cite{Y91} for $d\ge3$, by Genoud~\cite{G11} for $d=2$, and by Toland~\cite{T84} for $d=1$. The ground state gives the standing-wave solution

\begin{equation}
	u(t,x)=e^{it}Q(x)\notag
\end{equation}
whose critical spacetime norm is infinite in both time directions.  

The mass threshold is characterized by a sharp weighted Gagliardo--Nirenberg inequality due to Genoud~\cite{G12}:

\begin{equation}
	\int_{\mathbb{R} ^d} |x|^{-b}|f(x)|^{\frac{4-2b}{d} +2}dx\le \frac{d+2-b}{d}\left(\frac{ \|f\|_{L^2} }{ \|Q\|_{L^2} }\right)^{\frac{4-2b}{d}}\int _{\mathbb{R} ^d}|\nabla f(x)|^2dx. \label{eq:sharp}
\end{equation}
For an $H^1$ solution, the conserved energy is defined by
\begin{equation} 
	E(f)=\frac12\int_{\R^d}|\nabla f(x)|^2\dd x
	-\frac{d}{2(d+2-b)}\int_{\R^d}|x|^{-b}|f(x)|^{2+\frac{4-2b}{d}}\dd x.\notag
\end{equation}
Combining this identity with \eqref{eq:sharp}, we obtain
\begin{equation}\label{eq:energy-nonnegative}
	E(f)\ge\frac12\left[1-
	\left(\frac{\|f\|_{L_x^2}}{\|Q\|_{L_x^2}}\right)^{\frac{4-2b}{d}}\right]
	\|\nabla f\|_{L_x^2}^2.
\end{equation}

If $M(u_0)<M(Q)$, conservation of mass and energy therefore gives a uniform bound on the kinetic energy, and hence global existence for $H^1$ data~\cite{G12}. Such a bound alone, however, does not establish scattering. By adapting the concentration-compactness/rigidity strategy of Kenig and Merle~\cite{KM06} and long-time Strichartz methods developed by Dodson \cite{D12}, Liu, Miao and Zheng~\cite{LMZ25} proved global well-posedness and scattering for all $L^2$ data satisfying $M(u_0)<M(Q)$ in dimensions $d\ge3$ with  $0<b<\min\{2,d/2\}$. 
On the other hand, finite-time blowup for the focusing mass-critical INLS has also been extensively studied. Dinh~\cite{Di18} established finite-time blowup for negative-energy initial data under radiality or finite-variance assumptions. Subsequently, Cardoso and Farah~\cite{CF22} removed these restrictions, proving finite-time blowup for general negative-energy $H^1$ initial data.

At $M(u_0)=M(Q)$, the coercive bound \eqref{eq:energy-nonnegative} degenerates. The standing wave has zero energy and does not scatter; moreover, pseudoconformal invariance produces explicit solutions with the same mass that blow up in finite time. Combet and Genoud~\cite{CG16} established the corresponding finite-time classification, together with the variational characterization of the zero-energy case. We record the results in the form needed below. 

\begin{proposition}[Zero energy and finite-time classification, {\cite{CG16}}]\label{prop:classification}
	Let $d\ge3$ and $0<b<2$. If $f\in H^1$, $M(f)=M(Q)$, and $E(f)=0$, then
	\begin{equation} 
		f(x)=e^{i\theta}\lambda^{d/2}Q(\lambda x)\notag
	\end{equation}
	for some $\theta\in\R$ and $\lambda>0$.
	Moreover, if a maximal-lifespan $H^1$ solution with mass
	$M(Q)$ has a finite forward lifespan endpoint $T$,
	then it  must be of the form
	\begin{equation}\label{eq:pseudoconformal}
		u(t,x)=e^{i\theta}e^{i\lambda^2/(T-t)}
		e^{-i|x|^2/[4(T-t)]}
		\left(\frac{\lambda}{T-t}\right)^{d/2}
		Q\left(\frac{\lambda x}{T-t}\right),
		\qquad t<T,
	\end{equation}
	for some $\theta\in\mathbb R$ and $\lambda>0$.
 The corresponding classification for solutions blowing up at a finite backward endpoint follows by applying the above result to the time-reversed solution $\overline{u(-t,x)}$.
\end{proposition}

The finite-time classification of Combet and Genoud requires neither radiality nor a finite-variance assumption. Subsequently, Campos and Cardoso~\cite{CC22} established $L^2$ mass concentration for finite-time blowup solutions and provided an alternative proof of the minimal-mass blowup classification, following the compactness approach of Hmidi and Keraani~\cite{HK05}. 

The finite-time classification does not, by itself, resolve the rigidity problem for forward-global nonscattering solutions at the ground-state mass. For initial data with finite variance, the pseudoconformal transformation allows this problem to be reduced to the finite-time classification, as shown in Proposition~\ref{prop:finite-variance} below. For general $H^1$ data, however, such a reduction is not directly available, since the pseudoconformal transformation need not preserve $H^1$ regularity without the finite-variance assumption. Our main theorem establishes the desired rigidity for $\frac12<b<\min\{2,d/2\}$ without assuming either radial symmetry or finite variance.

We now state our main results at the threshold $M(u_0)=M(Q)$.
 
\begin{theorem}[Rigidity of forward-global threshold solutions]\label{thm:main}
	Let   $d\ge3,  \frac12<b<\min\left\{2,\frac d2\right\},$ 
 and let $u\in C([0,\infty);H^1(\R^d))$ be a strong solution of \eqref{eq:inls} with $M(u)=M(Q)$. If $u$ fails to scatter forward in $L^2$, then there exist constants $\theta\in\R$ and $\lambda>0$ such that
	\begin{equation}\label{eq:soliton}
		u(t,x)=e^{i\theta}e^{i\lambda^2t}\lambda^{d/2}Q(\lambda x)
		\qquad(t\ge0).
	\end{equation}
\end{theorem}

\begin{remark}
	The restriction $b>1/2$ ensures the convergence of the time integral in \eqref{eq:outer-smallb}. 
\end{remark}

\begin{remark}
	It is natural to ask whether Theorem~\ref{thm:main} remains valid for initial data merely in $L^2(\R^d)$. For the classical mass-critical NLS, Dodson~\cite{D24,D23} established $L^2$ rigidity in dimensions $1\le d\le15$ by developing long-time Strichartz estimates for solutions close to the ground state. His approach relies in part on the smoothness and rapid decay of the ground state. In contrast, the ground state of \eqref{eq:inls} lacks the same high-order regularity due to the singular coefficient $|x|^{-b}$. Consequently, this approach cannot be directly applied to the inhomogeneous equation, and removing the $H^1$ assumption remains an interesting open problem.
\end{remark}

 Since both the standing waves described by \eqref{eq:soliton} and the finite-time blowup solutions in \eqref{eq:pseudoconformal} are radial, their classification immediately leads to a scattering result for nonradial threshold data.

\begin{corollary}[Scattering for nonradial data]\label{cor:nonradial}
	Let $d\ge3$, $\frac12<b<\min\{2,d/2\}$, and suppose that $u_0\in H^1(\R^d)$ satisfies $M(u_0)=M(Q)$. If $u_0$ is not radial, then the corresponding maximal-lifespan solution is global in both time directions, and there exist $u_\pm\in H^1(\R^d)$ such that
	\begin{equation}\label{eq:H1scattering}
		\lim_{t\to\pm\infty}\|u(t)-e^{it\Delta}u_\pm\|_{H_x^1}=0.
	\end{equation}
\end{corollary}
Here radiality is understood with respect to the origin.
This conclusion reflects a structural difference from the
homogeneous mass-critical NLS, whose translation and Galilean
symmetries generate nonradial threshold solutions that do not
scatter. For \eqref{eq:inls}, these symmetries are absent, and
the classification shows that all exceptional threshold
solutions in the above parameter range are radial about
the origin.

Together with the finite-time classification of Proposition \ref{prop:classification}, the rigidity theorem also describes all possible maximal-lifespan threshold dynamics.
\begin{corollary}[Classification of critical-mass energy-space solutions]\label{cor:trichotomy}
	Let $d\ge3$ and $\frac12<b<\min\{2,d/2\}$. Every maximal-lifespan $H^1$ solution of \eqref{eq:inls} with mass $M(Q)$ satisfies exactly one of the following alternatives:
	\begin{enumerate}
		\item It is global and scatters in $H^1$ in both time directions.
		\item It is a global standing wave of the form \eqref{eq:soliton}.
		\item Its lifespan is a half-line with one finite endpoint. It is a pseudoconformal ground state of the form \eqref{eq:pseudoconformal}, or its time reversal, and scatters in $H^1$ at the infinite endpoint.
	\end{enumerate}
\end{corollary}

For $0<b\le\frac12$, we can obtain analogous rigidity and scattering results to those in Theorem~\ref{thm:main} under the additional assumption of finite variance, by exploiting the pseudoconformal transformation. In fact, this approach applies throughout the full parameter range
\begin{equation}\label{eq:full-range}
	d\ge3,\qquad 0<b<\min\left\{2,\frac d2\right\}.
\end{equation}
To state the results, we introduce the finite-variance space
\begin{equation}
	\Sigma(\R^d)=\{f\in H^1(\R^d):xf\in L^2(\R^d;\C^d)\},
	\qquad \|f\|_\Sigma=\|f\|_{H_x^1}+\|xf\|_{L_x^2}.\notag
\end{equation}

\begin{proposition}[Finite-variance rigidity throughout the full range]
	\label{prop:finite-variance}
	Assume \eqref{eq:full-range}, and let $u\in C([0,\infty);H^1(\R^d))$ be a strong solution of \eqref{eq:inls} with $u_0\in\Sigma$ and $M(u_0)=M(Q)$. Then either $u$ scatters forward in $H^1$ or it is a standing wave of the form \eqref{eq:soliton}. In particular, every such solution that fails to scatter forward in $L^2$ is a standing wave.  
\end{proposition}

The proof of Proposition~\ref{prop:finite-variance}, given in Subsection~\ref{subsec:finite-variance}, uses pseudoconformal invariance and the finite-time classification  results of Proposition \ref{prop:classification}.  Combining this finite-variance result with Corollary~\ref{cor:nonradial}, we obtain the following criterion. 
\begin{corollary}[Combined nonradial scattering criterion]
	\label{cor:combined-scattering}
	Assume \eqref{eq:full-range}, and let $u_0\in H^1(\R^d)$ satisfy $M(u_0)=M(Q)$. Suppose that $u_0$ is not radial. If
	\begin{equation} 
		b>\frac12\qquad\text{or}\qquad xu_0\in L^2,\notag
	\end{equation}
	then the corresponding maximal-lifespan solution is global in both time directions and scatters in $H^1$ in the sense of \eqref{eq:H1scattering}.  
\end{corollary}

\subsection{Outline of the Proof of the Main Theorem}\label{subsec:method-comparison}

We now outline the proof of Theorem~\ref{thm:main}, following the general strategy of Killip, Li, Visan and Zhang~\cite{KLVZ09}. By the sharp Gagliardo--Nirenberg inequality \eqref{eq:sharp}, we have $E(u)\ge0$ whenever $M(u)=M(Q)$. Moreover, Proposition~\ref{prop:classification} shows that $E(u)=0$ forces $u$ to be a standing wave. Thus, it suffices to rule out the possibility $E(u)>0$. Formally, the virial identity gives
\[
\frac{d^2}{dt^2}\int_{\R^d}|x|^2|u(t,x)|^2\dd x=16E(u).
\]
Since we do not assume finite variance, we instead consider a truncated virial functional $V_R(t)$, for which
\[
\begin{aligned}
	|V_R''(t)-16E(u)|
	\lesssim{}& R^{-2}M(u)
	+\int_{|x|>R}|\nabla u(t,x)|^2\dd x +\int_{|x|>R}|x|^{-b}
	|u(t,x)|^{2+\frac{4-2b}{d}}\dd x.
\end{aligned}
\]

If the exterior kinetic and potential-energy tails can be made sufficiently small on $[0,T]$ with a cutoff radius $R$ of order $\sqrt{1+T}$, then $V_R''(t)$ remains uniformly positive. Integrating in time would force $V_R'(T)-V_R'(0)$ to grow linearly in $T$. On the other hand, the phase--energy inequality in Proposition~\ref{prop:banica} gives
\[
|V_R'(t)|\lesssim R\sqrt{E(u)M(u)}.
\]
Since $R\sim\sqrt{1+T}$, these two bounds are incompatible as $T\to\infty$.

To obtain this spatial localization, we first establish a compactness property of the solution. Let $u\in C([0,\infty);H^1(\R^d))$ be a forward-global solution with $M(u)=M(Q)$ that fails to scatter forward in $L^2$. Following the concentration-compactness framework of Kenig and Merle~\cite{KM06} and adapting the argument in \cite[Section~3.1]{LMZ25}, we show in Section~\ref{sec:compactness} that $u$ is almost periodic modulo scaling. This provides the spatial mass localization needed for the subsequent kinetic-energy localization argument. We now recall the precise definition below.

\begin{definition}[Almost periodicity modulo scaling, {\cite{KLVZ09,LMZ25}}]\label{def:AP}
	A solution is said to be almost periodic modulo scaling on $[0,\infty)$ if there exists a function $N(t)>0$ such that
	\begin{equation} 
		\left\{N(t)^{-d/2}u\!\left(t,\frac{x}{N(t)}\right):t\ge0\right\}\notag
	\end{equation}
	is precompact in $L^2$. Equivalently, for every $\eta>0$, there exists $C(\eta)>0$ such that
	\begin{equation}\label{eq:APtails}
		\|u(t)\|_{L_x^2(|x|>C(\eta)/N(t))}
		+\|\widehat u(t)\|_{L_\xi^2(|\xi|>C(\eta)N(t))}<\eta
		\qquad(t\ge0).
	\end{equation}
\end{definition}

The function $N(t)$ represents the characteristic frequency scale of the solution, with $N(t)^{-1}$ corresponding to its spatial scale. While almost periodicity provides spatial localization of the mass, it does not directly yield the kinetic-energy localization required for the localized virial argument. To carry out the rigidity argument, we shall use two standard properties of almost-periodic solutions. 

The first is the local constancy of $N(t)$, which provides control of the spatial scale in the final virial argument.

\begin{proposition}[Local constancy, {\cite{KV13,LMZ25}}]\label{prop:localconstancy}
	For a nonzero almost-periodic solution, the scale function can be chosen to be continuous and satisfies
	\begin{equation}\label{eq:localconstancy}
		N(s)\simeq_u N(t)
		\quad\text{if }s,t\ge0,\quad |s-t|\le\delta_uN(t)^{-2},
	\end{equation}
	for some $\delta_u>0$. Consequently, \eqref{eq:localconstancy} implies
	\begin{equation}\label{eq:virial-scale-control}
		\sqrt{1+N(t)^{-2}}\le C_u\sqrt{1+t}
		\qquad \forall t\ge0.
	\end{equation}
\end{proposition}

The second is a reduced Duhamel formula, which serves as the starting point for the exterior frequency-decay estimates needed to establish kinetic-energy localization.

\begin{proposition}[Forward reduced Duhamel formula, {\cite{KLVZ09,KV13,TVZ08}}]\label{prop:reduced}
	If an $H^1$ strong solution of \eqref{eq:inls} is almost periodic modulo scaling on $[0,\infty)$, then, for every fixed $t\ge0$,
	\begin{equation}\label{eq:reduced}
		u(t)=-i\wlim_{T\to\infty}
		\int_0^T e^{-i\tau\Delta}
		\left(|x|^{-b}|u(t+\tau)|^{\frac{4-2b}{d}}u(t+\tau)\right)\dd\tau
		\quad\text{in }L^2.
	\end{equation}
\end{proposition}
  To derive the required exterior frequency-decay estimates,
we develop a nonradial incoming/outgoing decomposition
using smooth angular cutoffs in phase space, inspired by
the approach of Killip, Li, Visan and Zhang~\cite{KLVZ09}.
The key is to estimate the nonlinear forcing separately
near the singularity at the origin and in the exterior region.

In dimensions three and four, the singular coefficient $|x|^{-b}$ prevents us from estimating the interior nonlinear forcing using mass conservation alone, as in the approach of Killip, Li, Visan and Zhang~\cite{KLVZ09}. The main difficulty arises from the singular coefficient $|x|^{-b}$, which prevents us from obtaining the required spatial integrability. Instead, we integrate by parts in time using the equation, which produces a divergence-form commutator and allows derivatives to fall on the propagation kernel rather than the singular coefficient.  See Section~\ref{sec:exterior} for details. 

For the exterior contribution, we exploit the spatial decay of $|x|^{-b}$ along the outgoing propagation region. This additional decay enables us to implement a frequency bootstrap argument. More precisely, assuming an exterior frequency-decay bound of order $\nu^{-s}$ for some $0\le s<1$, we combine the dual Strichartz estimates with the decay of the spatial weight to improve the exponent from $s$ to $s+2\varepsilon$.
 
Starting from $s=0$, which follows from mass conservation, we iterate this improvement finitely many times to obtain, for every fixed $r>0$,
\begin{equation}\label{eq:exterior-intro}
	\sup_{t\ge0}\|\varphi_{>r}P_\nu u(t)\|_{L_x^2}
	\le C_{u,r}\nu^{-1-\varepsilon}+C\|P_\nu u_0\|_{L_x^2},
	\qquad \nu\ge1,
\end{equation}
where
\begin{equation}\label{eq:epsilon}
	\varepsilon=
	\begin{cases}
		\dfrac b2-\dfrac14,&\dfrac12<b\le1,\\[3pt]
		\dfrac b4,&1\le b<2.
	\end{cases}
\end{equation}
The details of the frequency bootstrap argument and the proof of \eqref{eq:exterior-intro} are given in Section~\ref{sec:exterior}.

In Subsection~\ref{Kinetic-energy localization}, we combine \eqref{eq:exterior-intro} with the spatial mass localization \eqref{eq:APtails} and the $H^1$ regularity of the initial data to establish kinetic-energy localization, following the argument of \cite{KLVZ09}. A localized weighted Gagliardo--Nirenberg inequality then yields the corresponding localization of the nonlinear potential energy. These estimates allow us to control the error terms in the truncated virial identity.

Then in Subsection~\ref{sec:virial}, we complete the rigidity argument. Combining the energy localization with the scale bound \eqref{eq:virial-scale-control} and the phase--energy inequality of Banica~\cite{B04}, we obtain the virial contradiction outlined above whenever $E(u)>0$. Consequently, $E(u)=0$, and Proposition~\ref{prop:classification} yields \eqref{eq:soliton}. Unlike the radial argument in \cite[Section~4]{KLVZ09}, our proof does not require a separate analysis of different frequency-scale regimes.

 Finally, we treat the finite-variance case without using the exterior frequency-decay estimates. The pseudoconformal transformation maps the infinite forward time interval to a finite interval. If the transformed solution extends beyond the terminal time, the original solution scatters forward; otherwise, the finite-time blowup classification of Combet and Genoud~\cite{CG16} implies that the original solution is a standing wave. This proves Proposition~\ref{prop:finite-variance}.  

\subsection{Organization of the paper}

Section~\ref{sec:prelim} fixes the notation and recalls Lorentz-space inequalities, Strichartz estimates, and the local and stability theory. Section~\ref{sec:compactness} establishes almost periodicity modulo scaling for forward-global nonscattering solutions at the ground-state mass. In Section~\ref{sec:exterior} we prove the exterior frequency-decay estimate using the nonradial incoming/outgoing decomposition, an interior time integration by parts, and an exterior frequency bootstrap. Section~\ref{sec:kinetic} combines kinetic-energy localization with the localized virial argument to obtain rigidity. Finally, Section~\ref{sec:conclusion} deduces the main classification and scattering results and proves the finite-variance statements by pseudoconformal transformation.

\section{Preliminaries}\label{sec:prelim}

	 In this section, we collect some notation and preliminary results that will be used throughout the paper. We begin by introducing the notation and recalling several Bernstein and mismatch estimates. We then review the definition and basic properties of Lorentz spaces, which are useful for handling the singularity of the inhomogeneous coefficient. Finally, we recall the Strichartz estimates, local well-posedness, and stability results for the Cauchy problem \eqref{eq:inls}.

	\subsection{Some notations}
	We use the standard notation for mixed Lebesgue space-time norms and Sobolev spaces. 
	We write  $A\lesssim  B$ to denote  $A\le CB$ for some  $C>0$. If  $A\lesssim B$ and  $B\lesssim A$, then we write  $A\approx B$.  We write  $A\ll B$   to denote  $A\le cB$ for some  small $c>0$.  If  $C$ depends upon some additional parameters, we will indicate this with subscripts; for example,  $X\lesssim _u Y$ denotes that  $X\le C_u Y$ for some  $C_u$ depending on  $u$.    We use  $O(Y)$ to denote any quantity  $X$  such that  $|X|\lesssim  Y$. We use     $\langle x \rangle $ to denote  $(1+|x|^2)^{\frac{1}{2}}$.        We write $L_t^q L_x^r$ to denote the Banach space with norm
	\[
	\|u\|_{L_t^q L_x^r (\mathbb{R} \times \mathbb{R}^d)} := \left( \int_{\mathbb{R}} \left( \int_{\mathbb{R}^d} |u(t, x)|^r \, dx \right)^{q/r} \, dt \right)^{1/q},
	\]
	with the usual modifications when $q$ or $r$ are equal to infinity, or when the domain $\mathbb{R} \times \mathbb{R}^d$ is replaced by spacetime slab such as $I \times \mathbb{R}^d$. When $q = r$ we abbreviate $L_t^q L_x^q$ as $L_{t,x}^q$.

	Fix a nonincreasing smooth radial function $\varphi$ with
	\[
	0\le\varphi\le1,\qquad
	\varphi(x)=1\ (|x|\le1),\qquad
	\varphi(x)=0\ (|x|\ge25/24).
	\]
	Set
	\begin{equation} 
		\varphi_{\le R}(x)=\varphi(x/R),\qquad
		\varphi_{>R}=1-\varphi_{\le R},\qquad
		\widehat{P_{\le\nu}f}(\xi)=\varphi(\xi/\nu)\widehat f(\xi),\notag
	\end{equation}
	\[
	P_\nu=P_{\le\nu}-P_{\le\nu/2},\qquad P_{>\nu}=1-P_{\le\nu}.
	\]
	All frequency sums are dyadic. A smooth enlarged annular multiplier $\widetilde P_\nu$ is chosen so that $\widetilde P_\nu P_\nu=P_\nu$. The fixed dyadic frequency $\nu$ is distinct from the almost-periodic frequency scale $N(t)$.

 We recall the following Bernstein and mismatch estimates, which will be used in Sections~\ref{sec:exterior} and \ref{sec:kinetic}.
 
 \begin{lemma}[Bernstein estimates, {\cite{KLVZ09}}]\label{lem:bernstein}
 	For $1\le p\le q\le\infty$, $\sigma\ge0$, and $\nu>0$,
 	\begin{align}
 		\|P_\nu f\|_{L_x^q}
 		&\lesssim\nu^{d(1/p-1/q)}\|P_\nu f\|_{L_x^p},\notag\\
 		\||\nabla|^\sigma P_\nu f\|_{L_x^p}
 		&\lesssim\nu^\sigma\|P_\nu f\|_{L_x^p},
 		\qquad
 		\|P_\nu f\|_{L_x^p}
 		\lesssim\nu^{-1}\|\nabla f\|_{L_x^p}.\notag
 	\end{align}
 \end{lemma}
 
 We also recall the following mismatch estimates for spatial and frequency localization.
 
 \begin{lemma}[Mismatch estimates, {\cite[Lemmas~2.1--2.3]{KLVZ09}}]\label{lem:mismatch}
 	For every integer $m\ge0$,
 	\begin{align}
 		\|\varphi_{>R}P_{\le\nu}(\varphi_{\le R/2}f)\|_{L_x^p}
 		&\le C_m(\nu R)^{-m}\|f\|_{L_x^p},\notag\\
 		\|\varphi_{>R}\nabla P_{\le\nu}(\varphi_{\le R/2}f)\|_{L_x^p}
 		&\le C_m\nu(\nu R)^{-m}\|f\|_{L_x^p}.\notag
 	\end{align}
 	These estimates also hold for enlarged annular multipliers. If
 	$\max\{\nu,\mu\}\ge4\min\{\nu,\mu\}$, then
 	\begin{align}
 		\|P_\nu\varphi_{<R}P_\mu f\|_{L_x^2}
 		&\le C_m(R\max\{\nu,\mu\})^{-m}\|f\|_{L_x^2},\notag\\
 		\|P_\nu\varphi_{<R}\nabla P_\mu f\|_{L_x^2}
 		&\le C_m\mu(R\max\{\nu,\mu\})^{-m}\|f\|_{L_x^2}.\notag
 	\end{align}
 	The same estimates hold if we replace $\varphi_{\le R}$ by $\varphi_{>R}$.
 \end{lemma}

	\subsection{Lorentz spaces}
	 In this subsection, we recall the definition of Lorentz spaces and some of their relevant properties.

	Let $f$ be a measurable function on $\mathbb{R}^d$. The distribution function of $f$ is defined by
	\begin{equation}
		d_f(\lambda)=: |\{x\in \mathbb{R}^d : |f(x)|>\lambda\}|, \quad \lambda>0,\notag
	\end{equation}
	where $|A|$ is the Lebesgue measure of a set $A$ in $\mathbb{R}^d$. The decreasing rearrangement of $f$ is defined by
	\begin{equation}
		f^*(s)=: \inf \left\{ \lambda>0 : d_f(\lambda)\leq s\right\}, \quad s>0.\notag
	\end{equation}
	
	\begin{defn}[Lorentz spaces] ~\\
		Let $0<r<\infty$ and $0<\rho\leq \infty$. The Lorentz space $L^{r,\rho}(\mathbb{R}^d)$ is defined by
		\begin{equation}
			L^{r,\rho}(\mathbb{R}^d)=: \left\{ f \text{ is measurable on } \mathbb{R}^d : \|f\|_{L^{r,\rho}}<\infty\right\}, \notag
		\end{equation}
		where
		\[
		\|f\|_{L^{r,\rho}}=: \left\{
		\begin{array}{cl}
			( \frac{\rho}{r} \int_0^\infty (s^{1/r} f^*(s))^\rho \frac{1}{s}ds)^{1/\rho} &\text{ if } \rho <\infty, \\
			\sup_{s>0} s^{1/r} f^*(s) &\text{ if } \rho=\infty.
		\end{array}
		\right.
		\]
	\end{defn}
	
	We collect the following basic properties of $L^{r,\rho}(\mathbb{R}^d)$ in the following lemmas.
	\begin{lemma}[Properties of Lorentz spaces \cite{ON63}] \ 
		\begin{itemize}
			\item For $1<r<\infty$, $L^{r,r}(\mathbb{R}^d) \equiv L^r(\mathbb{R}^d)$ and by convention, $L^{\infty,\infty}(\mathbb{R}^d)= L^\infty(\mathbb{R}^d)$.
			\item For $1<r<\infty$ and $0<\rho_1<\rho_2\leq \infty$, $L^{r,\rho_1}(\mathbb{R}^d)\subset L^{r,\rho_2}(\mathbb{R}^d)$. 
			\item For $1<r<\infty$, $0<\rho \leq \infty$, and $\theta>0$, $\||f|^\theta\|_{L^{r,\rho}} = \|f\|^\theta_{L^{\theta r, \theta \rho}}$.
			\item For $b>0$, $|x|^{-b} \in L^{\frac{d}{b},\infty}(\mathbb{R}^d)$ and $\||x|^{-b}\|_{L^{\frac{d}{b},\infty}} = |B(0,1)|^{\frac{b}{d}}$, where $B(0,1)$ is the unit ball of $\mathbb{R}^d$.
		\end{itemize}
	\end{lemma}
	\begin{lemma}[H\"older's inequality \cite{ON63}]  \ 
		\begin{itemize}
			\item  Let $1<r, r_1, r_2<\infty$ and $1\leq \rho, \rho_1, \rho_2 \leq \infty$ be such that
			\[
			\frac{1}{r}=\frac{1}{r_1}+\frac{1}{r_2}, \quad \frac{1}{\rho} \leq \frac{1}{\rho_1}+\frac{1}{\rho_2}.
			\]
			Then for any $f \in L^{r_1, \rho_1}(\mathbb{R}^d)$ and $g\in L^{r_2, \rho_2}(\mathbb{R}^d)$
			\[
			\|fg\|_{L^{r,\rho}} \lesssim \|f\|_{L^{r_1, \rho_1}} \|g\|_{L^{r_2,\rho_2}}.
			\]
			\item Let  $1<r_1,r_2<\infty $ and  $1\le \rho_1,\rho_2\le \infty $  be such that 
			\begin{equation}
				1=\frac{1}{r_1}+\frac{1}{r_2},\quad 1\le \frac{1}{\rho_1}+\frac{1}{\rho_2}.\notag
			\end{equation}
			Then  for any $f \in L^{r_1, \rho_1}(\mathbb{R}^d)$ and $g\in L^{r_2, \rho_2}(\mathbb{R}^d)$
			\begin{equation}
				\|fg\|_{L^1}\lesssim  \|f\|_{L^{r_1, \rho_1}} \|g\|_{L^{r_2,\rho_2}}.\notag
			\end{equation}
		\end{itemize}
	\end{lemma}
	
 \subsection{Strichartz estimates, local well-posedness and stability}

In this subsection, we recall Strichartz estimates in Lorentz spaces, together with the local well-posedness and stability results for the Cauchy problem \eqref{eq:inls}.
 \begin{defn}[Admissibility] 
 	A pair $(q,r)$ is said to be Schr\"odinger admissible, for short $(q,r)\in \Lambda $, where
 	\begin{equation}
 		\Lambda=\left\{(q,r): 2\le q,r\le \infty,\  	\frac{2}{q}+\frac{d}{r}=\frac{d}{2},\ (q,r,d)\neq (2,\infty ,2)\right\}.\notag
 	\end{equation}
 \end{defn}

 \begin{proposition}[Strichartz estimates, \cite{KT98,LMZ25}]\label{prop:strichartz}\ \\
 	\begin{itemize}
 		\item Let $(q,r)\in \Lambda $ with $r<\infty$. Then for any $f\in L^2(\mathbb{R}^d)$
 		\begin{align}  
 			\|e^{it\Delta }f\|_{L^q_t L^{r,2}_x(\mathbb{R}\times \mathbb{R}^d)} \lesssim \|f\|_{L^2_x(\mathbb{R}^d)}.\notag
 		\end{align}
 		
 		\item Let $(q_1, r_1), (q_2,r_2)\in \Lambda $ with $r_1, r_2<\infty$, $t_0\in \mathbb{R}$ and $I\subset \mathbb{R}$ be an interval containing $t_0$. Then 		for any $F\in L_t^{q_2'}L^{r_2',2}_x(I\times \mathbb{R}^d)$ 
 		\begin{align} 
 			\left\|\int_{t_0}^t e^{i(t-\tau)\Delta } F(\tau) d\tau\right\|_{L^{q_1}_tL^{r_1,2}_x(I\times \mathbb{R}^d)} \lesssim \|F\|_{L^{q_2'}_tL^{r_2',2}_x(I\times \mathbb{R}^d)}.\notag
 		\end{align}
 		\item   Let  $(q,r)\in \Lambda$. Then  
 		\begin{equation} 
 			\left\|\int_I e^{-is\Delta}h(s)\dd s\right\|_{L_x^2}
 			\lesssim\|h\|_{L_s^{q'}L_x^{r'}(I\times\R^d)}.\notag
 		\end{equation}
 	 The constants are independent of $I$.  
 	\end{itemize}
 \end{proposition}
 
 \begin{proposition}[Local well-posedness, {\cite{AT21,CG16,LMZ25}}]\label{prop:local}
 	Let $d\ge3$ and $0<b<\min\{2,d/2\}$. For any $u_0\in L^2(\R^d)$ and $t_0\in\R$, there exists a unique maximal-lifespan strong solution $u:I\times\R^d\to\C$ to \eqref{eq:inls} with $u(t_0)=u_0$. The solution conserves its mass and satisfies the following properties:
 	
 	\begin{enumerate}[label=\textup{(\alph*)},leftmargin=2.8em]
 		\item \textit{Blowup criterion.}
 		If $\sup I<\infty$, then
 	  $ 		\crit{u}{[t_0,\sup I)}=\infty.$ 
 		An analogous statement holds at the left endpoint of $I$.
 		
 		\item \textit{Scattering criterion.}
 		If  $ 		\crit{u}{[t_0,\sup I)}<\infty,$ 
 		then $\sup I=\infty$ and $u$ scatters forward in $L^2$. Conversely, forward $L^2$ scattering implies that the forward critical spacetime norm is finite. Analogous statements hold backward in time.
 		
 		\item \textit{Continuous dependence.}
 		If $u_{0,n}\to u_0$ in $L^2(\R^d)$ and $u_n$ denotes the corresponding maximal-lifespan solution with initial data $u_{0,n}$ at $t_0$, then $u_n\to u$ locally uniformly in $L^2$ on the lifespan of $u$.
 		
 		\item \textit{Wave operators.}
 		For every $\phi\in L^2(\R^d)$, there exist $T>0$ and a unique strong solution $u$ to \eqref{eq:inls} on $(T,\infty)$ such that
 		\[
 		\lim_{t\to\infty}\|u(t)-e^{it\Delta}\phi\|_{L_x^2}=0.
 		\]
 		An analogous statement holds backward in time.
 		
 		\item \textit{Small-data scattering.}
 		If $\|u_0\|_{L_x^2}$ is sufficiently small, then $u$ is global and scatters in both time directions. Moreover,  $ 		\crit{u}{\R}\lesssim\|u_0\|_{L_x^2}.$ 
 
 		\item \textit{$H^1$ persistence and scattering.}
 		If $u_0\in H^1(\R^d)$, then $u\in C(I;H^1(\R^d))$ and its energy is conserved.  
 		If a maximal lifespan endpoint is finite, then the $H^1$ norm diverges as that endpoint is approached. Furthermore, if $u$ scatters in $L^2$ in either time direction, then it also scatters in $H^1$ in the same direction.
 	\end{enumerate}
 \end{proposition}

 Beyond the small-data theory, Liu, Miao and Zheng~\cite{LMZ25} established global well-posedness and scattering for the focusing mass-critical INLS with mass strictly below that of the ground state. We recall  the result below.
 
 \begin{proposition}[Scattering below the ground-state mass, {\cite[Theorem~1.1]{LMZ25}}]\label{prop:subthreshold-scattering}
 	Let $d\ge3$ and $0<b<\min\{2,d/2\}$. Suppose that $u_0\in L^2(\R^d)$ satisfies
 	\[
 	M(u_0)<M(Q).
 	\]
 	Then the corresponding maximal-lifespan solution to \eqref{eq:inls} is global and scatters in $L^2$ in both time directions.
 \end{proposition}

 We also recall the following long-time stability result, which will be used in the concentration-compactness argument in Section~\ref{sec:compactness}.
	\begin{proposition}[Long-time stability, {\cite[Proposition~2.3]{LMZ25}}]\label{prop:stability}
		Let $I$ contain $t_0$, and suppose an approximate solution $\widetilde u$ satisfies
		\[
		\|\widetilde u\|_{L_t^\infty L_x^2(I\times\R^d)}
		+\crit{\widetilde u}{I}\le A,
		\]
		and
		\[
		(i\partial_t+\Delta)\widetilde u+
		|x|^{-b}|\widetilde u|^{\frac{4-2b}{d}}\widetilde u=e.
		\]
		For every $A<\infty$ there is $\delta(A)>0$ such that, if
		\[
		\|u_0-\widetilde u(t_0)\|_{L_x^2}+\dual{e}{I}\le\delta\le\delta(A),
		\]
		then the exact solution with $u(t_0)=u_0$ exists throughout $I$, and
		\[
		\|u-\widetilde u\|_{L_t^\infty L_x^2(I\times\R^d)}
		+\crit{u-\widetilde u}{I}\le C(A)\delta.
		\]
	\end{proposition}
 
 \section{Compactness of nonscattering solutions at the ground-state mass}
 \label{sec:compactness}

In this section, following the concentration-compactness framework of Kenig and Merle~\cite{KM06} and the argument in \cite[Section~3.1]{LMZ25}, we establish the almost periodicity modulo scaling of forward-global nonscattering solutions at the ground-state mass. More precisely, we prove the following proposition.

 \begin{proposition}[Compactness at the ground-state mass]
 	\label{prop:compactness}
 	Let $d\ge3$ and $\frac12<b<\min\{2,d/2\}$. Suppose that
 	$u\in C([0,\infty);H^1(\R^d))$ solves \eqref{eq:inls},
 	$M(u)=M(Q)$, and $u$ does not scatter forward in $L_x^2$.
 	Then there is a continuous function $N:[0,\infty)\to(0,\infty)$
 	such that
 	\[
 	\left\{N(t)^{-d/2}u\!\left(t,\frac{x}{N(t)}\right):t\ge0\right\}
 	\]
 	is precompact in $L_x^2$. 
 \end{proposition}

The proof of Proposition~\ref{prop:compactness} relies on two ingredients. The first is the linear profile decomposition, which allows us to analyze bounded sequences in $L^2$ in terms of asymptotically orthogonal profiles. We recall this result below.

 	\begin{proposition}[Linear profile decomposition, {\cite{LMZ25}}]\label{prop:profiles}
 	Let $u_n$ be a bounded sequence in $L^2(\mathbb{R}^d)$.  Then after passing to a subsequence if necessary,		
 there exist $J^*\in\mathbb{N}\cup\{\infty\}$; profiles $\phi^j\in  L^2 \setminus\{0\}$; scales $\lambda_n^j\in(0,\infty)$; space translation parameters $x_n^j\in\mathbb{R}^d$;  frequency translation parameters $\xi_n^j\in\mathbb{R}^d$; time translation parameters $t_n^j$; and remainders $r_n^J$ such that the following  decomposition holds for $1\leq J\leq J^*$:
 	\begin{equation}\label{eq:PS-profile-decomposition}
 		u_n(x)=\sum_{j=1}^J(\lambda_n^j)^{-d/2}e^{ix\cdot\xi_n^j}
 		[e^{it_n^j\Delta}\phi^j]\!\left(\frac{x-x_n^j}{\lambda_n^j}\right)+r_n^J(x),
 	\end{equation}
 		here  $r_n^J\in L^2_x(\mathbb{R}^d)$ obey
 		 	\begin{equation}	\label{eq:PS-free-remainder}
 			\lim_{J\to J^*}\limsup_{n\to\infty}\crit{e^{it\Delta}r_n^J}{\R}=0.
 		\end{equation}
 			Moreover, for any  $j\neq k$,
 		\begin{equation} 
 		\lim_{n\to\infty} \biggl\{|\log\bigl[\tfrac{\lambda_n^j}{\lambda_n^k}\bigr] |+ \tfrac{|x_n^j-x_n^k|^2}{\lambda_n^j\lambda_n^k} +\lambda_n^j\lambda_n^k|\xi_n^j-\xi_n^k|^2+ \tfrac{|t_n^j(\lambda_n^j)^2-t_n^k(\lambda_n^k)^2|}{\lambda_n^j\lambda_n^k}\biggr\} = \infty.\notag
 		\end{equation}
 		In addition, we may assume that  either  $ t_n^j\equiv 0$ or  $t_n^j\to\pm\infty$, either  $  x_n^j\equiv 0$ or  $\frac{|x_n^j|}{\lambda _n^j}\to\infty$, and that either  $\xi_n^j\equiv0$ or  $\lambda_n^j\xi_n^j\to \infty $.   
 		Furthermore, for any  $J\ge1$ we have the  mass decoupling property
 	\begin{equation}	\label{eq:PS-mass-decoupling}
 		M(u_n)=\sum_{j=1}^JM(\phi^j)+M(r_n^J)+o_n(1),
 	\end{equation}
 \end{proposition}

 We also need the following result, which ensures the existence of scattering solutions for profiles whose spatial or frequency parameters escape to infinity.
 
 \begin{proposition}[Escaping profiles, {\cite[Proposition~2.4]{LMZ25}}]\label{prop:escape}
 	 Let   $x_n\in \mathbb{R}^d$ and $\xi_n\in\mathbb{R}^d$ satisfy
 	 \begin{equation}
 	 	\lim_{n\rightarrow \infty }\left(|x_n|+|\xi_n|\right)=\infty .\notag
 	 \end{equation}
 	 Let $\phi\in L^2(\mathbb{R}^d)$ and define
 	 \begin{equation}
 	 	\phi_n(x) =: e^{ix\cdot \xi_n}\phi (x-x_n). \notag
 	 \end{equation}
 	 Then there exists a subsequence $n_k$ and a global scattering solution $v_{n_k}$ to (\ref{eq:inls}) satisfying 
 	 \begin{equation}
 	 	v_{n_k}(0)=\phi _{n_k}\quad\text{and}\quad  \|v_{n_k}\|_{L_t^{\frac{2(d+4-2b)}{d}}L_x^{\frac{2(d+4-2b)}{d+2-2b},2}(\mathbb{R} \times \mathbb{R} ^d)}\lesssim 1,\notag
 	 \end{equation}
 	 where the implicit constant is independent of  $k$. 
 \end{proposition}

 With these two results at hand, we now prove Proposition~\ref{prop:compactness}.

\begin{proof}[Proof of Proposition~\ref{prop:compactness}]
	Let $t_n\to\infty$ be an arbitrary sequence and set $u_n(t)=u(t_n+t)$. Since $u$ does not scatter forward in $L_x^2$, we have
	\begin{equation}\label{eq:PS-blowup}
		\crit{u_n}{[0,\infty)}\longrightarrow\infty\quad\text{as}\quad n\rightarrow\infty .
	\end{equation}
	
	Applying Proposition~\ref{prop:profiles} to the bounded sequence $u_n(0)$ and passing to a subsequence, we obtain the linear profile decomposition \eqref{eq:PS-profile-decomposition}, together with \eqref{eq:PS-free-remainder} and \eqref{eq:PS-mass-decoupling}.
	
	Our goal is to show that exactly one profile survives, that its time, spatial, and frequency translation parameters can be taken to be zero, and that it carries the entire mass of the solution. We establish these facts in the following three steps.

	\begin{enumerate}
		\item There is exactly one nonzero profile in
		\eqref{eq:PS-profile-decomposition}, so that
		\begin{equation}\label{eq:singleprofile}
			u_n(0,x)=\lambda_n^{-d/2}e^{ix\cdot\xi_n}
			[e^{it_n\Delta}\phi]\!\left(\frac{x-x_n}{\lambda_n}\right)
			+r_n(x).
		\end{equation}
		\item The parameters in \eqref{eq:PS-profile-decomposition} satisfy $(t_n,x_n,\xi_n)\equiv (0,0,0)$.
		\item The remainder satisfies $\|r_n\|_{L_x^2}\to0$ and
		$M(\phi)=M(Q)$.
	\end{enumerate}

\medskip
\noindent\textit{Proof of (i).}
We first show that the decomposition contains exactly one nonzero profile.
	If there is no nonzero profile, then
	\eqref{eq:PS-free-remainder} implies
	\[
	\crit{e^{it\Delta}u_n(0)}{\R}\longrightarrow0.
	\]
	By the small-data global well-posedness and scattering theory in Proposition~\ref{prop:stability}, the solutions $u_n$ then have finite critical spacetime norms for all sufficiently large $n$, contradicting \eqref{eq:PS-blowup}. Hence $J^*\ge1$.
	
	Suppose that $J^*\ge2$. Since all the profiles are nonzero,
	\eqref{eq:PS-mass-decoupling} implies
	\[
	0<M(\phi^j)<M(Q)\qquad(1\le j\le J^*).
	\]
	
	We now construct nonlinear profiles corresponding to the linear profiles in
	\eqref{eq:PS-profile-decomposition}. We distinguish the following three cases.

\medskip
\noindent\textit{Case 1:}
$t_n^j\to\pm\infty$.

 By Proposition~\ref{prop:local}, we may choose a global scattering solution $v^j$ to \eqref{eq:inls} such that
 \[
 \|v^j(t)-e^{it\Delta}\phi^j\|_{L_x^2}\to0
 \quad\text{as }t\to\pm\infty.
 \]
 We then define
 \[
 v_n^j(t,x)
 =(\lambda_n^j)^{-d/2} e ^{ix\cdot \xi_n^j}
 v^j\left(\frac{t}{(\lambda_n^j)^2}+t_n^j,
 \frac{x-x_n^j}{\lambda_n^j}\right).
 \]
 
	\medskip
\noindent\textit{Case 2:}
$t_n^j\equiv0$, $x_n^j\equiv0$; and $\xi_n^j\equiv0$.

Let $v^j$ be the solution to \eqref{eq:inls} with initial data
$v^j(0)=\phi^j$. Since $M(\phi^j)<M(Q)$, Proposition~\ref{prop:subthreshold-scattering}
implies that $v^j$ is global and scatters in both time directions.
We define
\[
v_n^j(t,x)
=(\lambda_n^j)^{-d/2}
v^j\left(\frac{t}{(\lambda_n^j)^2},
\frac{x}{\lambda_n^j}\right).
\]

\noindent\textit{Case 3:}   $t_n^j\equiv0$; and  
$\frac{|x_n^j|}{\lambda_n^j}\to\infty$ or
$\lambda_n^j|\xi_n^j|\to\infty$.

After passing to a subsequence, Proposition~\ref{prop:escape} implies that, for all sufficiently large $n$, the solutions $w_n^j(t,x)$ to \eqref{eq:inls} with initial data
\[
w_n^j(0)=e^{ix\cdot\lambda_n^j\xi_n^j}
\phi^j\left(x-\frac{x_n^j}{\lambda_n^j}\right)
\]
exist globally and scatter. We then define
\[
v_n^j(t,x)
=(\lambda_n^j)^{-d/2}
e^{ix_n^j\cdot\xi_n^j}
w_n^j\left(
\frac{t}{(\lambda_n^j)^2},
\frac{x}{\lambda_n^j}
\right).
\]

	\medskip
	Combining the nonlinear profiles constructed above, we define
	\[
	u_n^J(t,x)
	=\sum_{j=1}^Jv_n^j(t,x)+e^{it\Delta}r_n^J(x).
	\]

	By the construction of $u_n^J$ and the profile decomposition \eqref{eq:PS-profile-decomposition}, we have
	\begin{equation}
		\lim_{n\rightarrow \infty }\|u_n^J(0)-u_n(0)\|_{L^2}=0.\notag
	\end{equation}
	Moreover, \cite[Lemma~3.1]{LMZ25} gives
	\begin{equation} 
		\limsup_{J\to J^*}\limsup_{n\to\infty}
		\left(\|u_n^J\|_{L_t^\infty L_x^2(\R\times\R^d)}
		+\crit{u_n^J}{\R}\right)\lesssim 1,\notag
	\end{equation}
	and
	\begin{equation} 
		\lim_{J\to J^*}\limsup_{n\to\infty}
		\dual{(i\partial_t+\Delta)u_n^J+
			|x|^{-b}|u_n^J|^{\frac{4-2b}{d}}u_n^J}{\R}=0.\notag
	\end{equation}
	Hence, we can apply the stability result in Proposition~\ref{prop:stability} to obtain
	\[
	\limsup_{n\to\infty}\crit{u_n}{\R}<\infty,
	\]
which	contradicts\eqref{eq:PS-blowup}. Thus $J^*=1$, and \eqref{eq:singleprofile} follows.

\medskip
\noindent\textit{Proof of (ii).}
Having established the existence of a unique nonzero profile, we now show that its time, spatial, and frequency translation parameters can all be eliminated.

	Since $J^*=1$, the decomposition \eqref{eq:PS-profile-decomposition} reduces to \eqref{eq:singleprofile}, where either $t_n\equiv0$ or $t_n\rightarrow \pm \infty$, either $x_n\equiv0$ or $\frac{|x_n|}{\lambda _n}\rightarrow \infty$, and either $\xi_n\equiv 0$ or $\lambda_n\xi_n\rightarrow\infty$.
	
	We first show that the time shifts can be taken to be zero. Suppose that $t_n\to\infty$. Then
	\begin{equation}
		e^{it\Delta }u_{n}(0,x)=\lambda_n^{-\frac{d}{2}}e^{-it|\xi_n|^2}e^{ix\cdot\xi_n}\left(e^{i\frac{t}{\lambda_n^2}\Delta }e^{it_n\Delta }\phi\right)(\frac{x-2\xi_n t-x_n}{\lambda_n}),\notag
	\end{equation}
	and the corresponding critical spacetime norm tends to zero:
	\begin{equation}
		\|e^{it\Delta }u_{0,n}\|_{L^ {\frac{2(d+4-2b)}{d}} _tL_x^{\frac{2(d+4-2b)}{d+2-2b},2}([0,\infty )\times \mathbb{R} ^d)}\lesssim  \|e^{it\Delta }\phi\|_{L^ {\frac{2(d+4-2b)}{d}}_tL_x^{ \frac{2(d+4-2b)}{d+2-2b},2}([t_n,\infty )\times \mathbb{R} ^d)}\rightarrow0 \notag
	\end{equation}
	The corresponding free evolutions thus provide suitable approximate solutions with uniformly bounded global critical spacetime norms for sufficiently large $n$. The stability result would then yield uniform bounds for the exact solutions, contradicting \eqref{eq:PS-blowup}. The case $t_n\rightarrow-\infty$ is analogous. Hence \eqref{eq:singleprofile} reduces to
	\begin{equation}
		u_n(0,x)=\lambda_n^{-\frac{d}{2}}e^{ix\cdot \xi_n} \phi(\frac{x-x_n}{\lambda_n})+r_n(x), \label{716w4}
	\end{equation}
	where either $x_n\equiv0$ or $\frac{|x_n|}{\lambda _n}\rightarrow \infty$, and either $\xi_n\equiv 0$ or $\lambda_n|\xi_n|\rightarrow\infty$.
	
	We next show that the spatial and frequency translation parameters can also be taken to be zero. Suppose that $\frac{|x_n|}{\lambda _n}\to\infty$ or $\lambda_n|\xi_n| \rightarrow\infty$. By Proposition~\ref{prop:escape}, there exist global scattering solutions $v_n$ to \eqref{eq:inls} with initial data
	$$v_n(0)= e^{i\lambda_n x\cdot\xi_n}\phi(x-\frac{x_n}{\lambda_n}).$$
	By scaling, $\lambda_n^{-\frac{d}{2}}v_n(\lambda_n^{-2}t,\lambda_n^{-1}x)$ are also global scattering solutions to \eqref{eq:inls}, with initial data $\lambda_n^{-\frac{d}{2}}e^{ix\cdot \xi_n} \phi(\frac{x-x_n}{\lambda_n})$.
	Applying the stability result in Proposition~\ref{prop:stability}, we again obtain bounds for $u_n$ that contradict \eqref{eq:PS-blowup}. Consequently, \eqref{716w4} reduces to
	\begin{equation}
		u_n(0,x)=\lambda_n^{-\frac{d}{2}}  \phi(\frac{x}{\lambda_n})+r_n(x).\notag
	\end{equation}

\medskip
\noindent\textit{Proof of (iii).}
It remains to show that the unique profile carries the full ground-state mass and that the remainder converges to zero in $L_x^2$.

	It sufficies to show that $M(\phi)=M(Q)$ and
	\begin{equation}
		\|r_n(x)\|_{L^2}\rightarrow0\quad\text{as}\quad  n\rightarrow+\infty .\notag
	\end{equation}
	Suppose otherwise. If $M(\phi)<M(Q)$, then Proposition~\ref{prop:subthreshold-scattering} implies that the solution $v$ to \eqref{eq:inls} with initial data $\phi$ scatters. Moreover,
	\begin{equation}
		v_n(t,x)=:\lambda_n^{-\frac{d}{2}}v\left(\frac{t}{\lambda_n^2},+\frac{x}{\lambda_n}\right)\notag
	\end{equation}
	provide suitable approximate solutions with uniformly bounded global critical spacetime norms for sufficiently large $n$. Applying the stability result in Proposition~\ref{prop:stability}, we obtain uniform bounds for $u_n$, contradicting \eqref{eq:PS-blowup}.
	Therefore $\|r_n\|_{L_x^2}\to0$ and $M(\phi)=M(Q)$.

Therefore, combining (i), (ii), and (iii) established above, we conclude that, for any sequence $t_n\rightarrow\infty$, after passing to a subsequence, there exist scales $\lambda_n>0$ and a function $\phi\in L_x^2$ with $M(\phi)=M(Q)$ such that
\begin{equation} 
	\lambda_n^{d/2}u(t_n,\lambda_n x)\longrightarrow\phi(x)
	\quad\text{in }L_x^2.\notag
\end{equation}
Thus, every sequence of times tending to infinity admits a subsequence along which the corresponding rescaled solutions converge strongly in $L_x^2$. Together with the strong continuity of $u$ on bounded time intervals, this yields precompactness modulo scaling, as claimed.
$   $ 
\end{proof}

	\section{Exterior frequency decay}\label{sec:exterior}
	
	In this section, we establish an exterior frequency decay estimate for solutions satisfying the forward reduced Duhamel formula \eqref{eq:reduced}. This estimate will play a key role in the kinetic-energy localization argument in Section~\ref{sec:kinetic}. The proof combines an incoming/outgoing decomposition of the free Schr\"odinger evolution with estimates for the nonlinear terms in the interior and exterior regions.
	
	Throughout this section, let $u\in C([0,\infty);H^1)$ be a strong solution of \eqref{eq:inls} satisfying \eqref{eq:reduced}. We do not assume $M(u)=M(Q)$ or $\sup_{t\ge0}\|\nabla u(t)\|_{L_x^2}<\infty$. The implicit constants may depend on the initial data and fixed exterior radii, but are independent of time.

	\begin{theorem}[Exterior frequency decay]\label{thm:exterior}
		Let $d\ge3 $ and $ \frac12<b<\min\left\{2,\frac d2\right\}$. For every fixed $r>0$,
		\begin{equation}\label{eq:exterior}
			\sup_{t\ge0}\|\varphi_{>r}P_\nu u(t)\|_{L_x^2}
			\le C_{u,r}\nu^{-1-\varepsilon}+C\|P_\nu u_0\|_{L_x^2},
			\qquad \nu\ge1,
		\end{equation}
		where $\varepsilon>0$ is given by \eqref{eq:epsilon}.
	\end{theorem}

The proof of Theorem~\ref{thm:exterior} is organized as follows. In the first subsection, we introduce the incoming/outgoing decomposition and establish the corresponding nonstationary propagation estimates. In the second subsection, we estimate the interior and exterior nonlinear contributions, using integration by parts in time for the former and a frequency bootstrap argument for the latter. Finally, in the third subsection, we combine these estimates to complete the proof of Theorem~\ref{thm:exterior}.

\subsection{Incoming/outgoing decomposition and nonstationary estimates}

We begin by decomposing the frequency-localized solution into incoming and outgoing components. This decomposition allows us to exploit the direction of free Schr\"odinger propagation when estimating the nonlinear Duhamel terms. Unlike the radial incoming/outgoing decompositions, the operators introduced below are defined for general, not necessarily radial, functions.

	Choose $\rho_\pm\in C^\infty([-1,1])$ satisfying
	\[
	0\le\rho_\pm\le1,\quad \rho_++\rho_-=1,\quad
	\rho_+(s)=0\ (s\le-1/2),\quad
	\rho_-(s)=0\ (s\ge1/2).
	\]
	With $p(\zeta)=\varphi(\zeta)-\varphi(2\zeta)$, define
	\begin{equation}\label{eq:Pi}
		\Pi_{\nu,r}^{\pm}f(x)
		=\frac{\varphi_{>r}(x)}{(2\pi)^{d/2}}
		\int e^{ix\cdot\xi}
		\rho_\pm\!\left(\frac{x\cdot\xi}{|x||\xi|}\right)
		p(\xi/\nu)\widehat f(\xi)\dd\xi.
	\end{equation}
	The expression is set to zero for $|x|\le r$. The frequency support lies in $\{\nu/2\le|\xi|\le25\nu/24\}$, so no singularity at $\xi=0$ is present. The operators need not be idempotent; their incoming/outgoing interpretation refers to their phase-space supports and the propagation directions used below. Exactly,
	\begin{equation}\label{eq:Pi-sum}
		\Pi_{\nu,r}^{+}+\Pi_{\nu,r}^{-}=\varphi_{>r}P_\nu.
	\end{equation}
	Let $\widetilde p$ be smooth, supported in a fixed enlarged annulus, and equal to one on $\supp p$. Replacing $p$ by $\widetilde p$ in \eqref{eq:Pi} defines $\widetilde\Pi_{\nu,r}^\pm$, and
	\begin{equation}\label{eq:Pi-factor}
		\Pi_{\nu,r}^\pm=\widetilde\Pi_{\nu,r}^\pm P_\nu.
	\end{equation}

	We first establish the $L^2$ boundedness of these operators, which will allow us to combine the decomposition with standard Strichartz estimates.
	
	\begin{lemma}[Uniform $L^2$ boundedness]\label{lem:PiL2}
		The operators $\widetilde\Pi_{\nu,r}^\pm$ are bounded on $L^2$, uniformly in $\nu,r>0$. Consequently,
		\begin{equation}\label{eq:PiL2}
			\|\Pi_{\nu,r}^\pm f\|_{L_x^2}\le C\|P_\nu f\|_{L_x^2}.
		\end{equation}
	\end{lemma}

	\begin{proof}
		The kernel of $\widetilde\Pi_{\nu,r}^\pm$ is
		\[
		\frac{\varphi_{>r}(x)}{(2\pi)^d}
		\int e^{i(x-y)\cdot\xi}
		\rho_\pm\!\left(\frac{x\cdot\xi}{|x||\xi|}\right)
		\widetilde p(\xi/\nu)\dd\xi.
		\]
		Set $\xi=\nu\zeta$. Applying $(1-\Delta_\zeta)^k$ to the amplitude and integrating by parts yields the kernel bound
		\[
		C_k\nu^d(1+\nu^2|x-y|^2)^{-k}.
		\]
		For $k>\frac{d}{2}$, the integral of this bound with respect to either $x$ or $y$ is uniformly bounded, independently of $\nu$ and $r$. Hence, Schur's test gives the asserted boundedness. Finally, the  identity \eqref{eq:Pi-factor} yields \eqref{eq:PiL2}.
	\end{proof}
	
	We next establish nonstationary kernel estimates for the incoming and outgoing propagators when the input is localized near the origin. These estimates, including bounds on the input derivatives of the kernels, will be used to control the interior nonlinear terms without requiring derivatives of the solution.
	
	Choose a radial $\eta\in C_c^\infty(\R^d)$ with $0\le\eta\le1$, equal to one on the unit ball and zero outside the ball of radius two. The interior input region will have radius $(r+\nu\tau)/64$, where the output radius $r$ and the dyadic frequency $\nu$ are fixed. The following lemma provides the required kernel estimates on a slightly enlarged input region.

	\begin{lemma}[Nonstationary kernel and its input derivatives]\label{lem:kernel}
		Let $\nu r\ge1$, and let $K_{\nu,r}^\pm(\tau;x,y)$ be the kernel of $\Pi_{\nu,r}^\pm e^{\mp i\tau\Delta}$. For $j=0,1$, every integer $m\ge1$, and every $L>d$,
		\begin{equation}\label{eq:kernel}
			|\nabla_y^jK_{\nu,r}^\pm(\tau;x,y)|
			\le C_{m,L}\nu^{d+j}(\nu r+\nu^2\tau)^{-m}
			(1+\nu|x-y|)^{-L}
		\end{equation}
		whenever $|y|\le(r+\nu\tau)/16$.  The constants are independent of $\nu,r,\tau,x,y$.
	\end{lemma}

	\begin{proof}
		It suffices to consider the plus sign, since the minus-sign case follows by an analogous argument. The corresponding kernel is
		\begin{align}
			K_{\nu,r}^+(\tau;x,y)
			=\frac{\varphi_{>r}(x)}{(2\pi)^d}
			\int e^{i((x-y)\cdot\xi+\tau|\xi|^2)}
			\rho_+\!\left(\frac{x\cdot\xi}{|x||\xi|}\right)p(\xi/\nu)\dd\xi.
		\notag
		\end{align}
		On the support of $\rho_+$, we have $x\cdot\xi\ge-|x||\xi|/2$. Consequently,
		\begin{align}
			|x+2\tau\xi|^2
			&\ge |x|^2+4\tau^2|\xi|^2-2\tau|x||\xi| \ge\tfrac14(|x|+2\tau|\xi|)^2.
			\label{eq:geometric}
		\end{align}
		Using $|\xi|\ge\nu/2$, $|x|\ge r$, and $|y|\le(r+\nu\tau)/16$, we further obtain
		\begin{equation}\label{eq:phase-separation}
			|x-y+2\tau\xi|
			\ge\frac{7}{16}(|x|+\nu\tau)
			\ge\frac14(|x|+\nu\tau).
		\end{equation}
		After the change of variables $\xi=\nu\zeta$, the phase becomes
		\[
		\Phi(\zeta)=\nu(x-y)\cdot\zeta+\tau\nu^2|\zeta|^2.
		\]
		It follows from \eqref{eq:geometric} and \eqref{eq:phase-separation} that
		\[
		c_0\Lambda\le|\nabla_\zeta\Phi|\le C_0\Lambda,
		\qquad |\partial_\zeta^2\Phi|\le C\Lambda,
		\qquad \partial_\zeta^\gamma\Phi=0\quad(|\gamma|\ge3),
		\]
		where  $\Lambda=\nu(|x|+\nu\tau)$. 
		Using
		\[
		\mathcal L=\frac{\nabla_\zeta\Phi}{i|\nabla_\zeta\Phi|^2}\cdot\nabla_\zeta,\qquad \mathcal L(e^{i\Phi})=e^{i\Phi}
		\]
		and integrating by parts $k$ times, we obtain
		\[
		|K_{\nu,r}^+(\tau;x,y)|\le C_k\nu^d\Lambda^{-k}.
		\]
		Differentiation with respect to $y_j$ introduces only the factor $-i\xi_j=-i\nu\zeta_j$ into the amplitude, so the same argument yields
		\[
		|\nabla_yK_{\nu,r}^+(\tau;x,y)|\le C_k\nu^{d+1}\Lambda^{-k}.
		\]
		These estimates concern derivatives of the kernel itself, before any input cutoff is differentiated.
		
		Finally, since $\nu r\ge1$ and in view of the input support, we have
		\[
		\Lambda\ge\nu r+\nu^2\tau,\qquad
		1+\nu|x-y|\le C\Lambda.
		\]
		Choosing an integer $k\ge m+L$ yields \eqref{eq:kernel}.
	\end{proof}

As a consequence of Lemma~\ref{lem:kernel} and Schur's test, we obtain the following localized estimates, which will be used to control the interior nonlinear terms.

 If $a_\tau$ and a vector field $v$ are supported in $|y|\le(r+\nu\tau)/32$, then
	\begin{align}
		\|\Pi_{\nu,r}^\pm e^{\mp i\tau\Delta}(a_\tau f)\|_{L_x^2}
		&\le C_m(\nu r+\nu^2\tau)^{-m}\|a_\tau\|_{L_x^\infty}\|f\|_{L_x^2},
		\label{eq:localized-L2}\\
		\|\Pi_{\nu,r}^\pm e^{\mp i\tau\Delta}\diver v\|_{L_x^2}
		&\le C_m\nu(\nu r+\nu^2\tau)^{-m}\|v\|_{L_x^2}.
		\label{eq:localized-div}
	\end{align}

\subsection{Estimates for the interior and exterior nonlinear terms}
\label{subsec:exterior-bootstrap}

We now estimate the nonlinear terms arising in the Duhamel representation. We treat the interior and exterior regions separately. In the interior region, we use the equation and integration by parts in time to exploit the nonstationary kernel estimates. In the exterior region, we combine the spatial decay of the inhomogeneous coefficient with frequency localization and Strichartz estimates.

Let 
	\begin{equation}\label{eq:eta}
		\eta_\tau(x)=\eta\!\left(\frac{64x}{r+\nu\tau}\right),\qquad \tau\ge0.
	\end{equation}
	We first establish a uniform estimate for the interior contribution.

	\begin{lemma}[Interior tail estimate]\label{lem:inner}
		Let $d\ge3$, $0<b<2$, and $u\in C([0,\infty);H^1)$ solve \eqref{eq:inls} with conserved mass. Suppose $\nu r\ge1$. For every integer $m\ge1$, $t\ge0$, and $0\le S\le T<\infty$,
		\begin{align}
			&\left\|\int_S^T\Pi_{\nu,r}^+e^{-i\tau\Delta}
			\left[\eta_\tau|x|^{-b}|u(t+\tau)|^{\frac{4-2b}{d}}u(t+\tau)\right]\dd\tau\right\|_{L_x^2} \le C_m\|u_0\|_{L_x^2}(\nu r+\nu^2S)^{-m}.
			\label{eq:innerplus}
		\end{align}
		For $0\le S\le T\le t$,
		\begin{align}
			&\left\|\int_S^T\Pi_{\nu,r}^-e^{i\tau\Delta}
			\left[\eta_\tau|x|^{-b}|u(t-\tau)|^{\frac{4-2b}{d}}u(t-\tau)\right]\dd\tau\right\|_{L_x^2} \le C_m\|u_0\|_{L_x^2}(\nu r+\nu^2S)^{-m}.
			\label{eq:innerminus}
		\end{align}
		The constant depends only on $d,m$ and the fixed cutoffs, not on $\nu,r,t,S,T$. In particular, the future interior integral converges strongly in $L^2$. This lemma does not require the reduced Duhamel formula or a mass threshold.
	\end{lemma}

\begin{proof}
	\textit{Step 1: the finite-time identity and commutator representation.}
	On every compact time interval, the energy-space equation implies
	$u\in C_tH_x^1\cap C_t^1H_x^{-1}$, as follows from the standard weighted nonlinear estimate in \cite[Lemma~3]{CG16}. For each fixed $\nu$, the frequency cutoff maps $H_x^{-1}$ continuously into $L_x^2$. Thus, by \eqref{eq:Pi-factor} and Lemma~\ref{lem:PiL2}, the following product rule is justified as an identity in $L_x^2$:
	\begin{align}
		&\frac{\dd}{\dd\tau}\left\{\Pi_{\nu,r}^+e^{-i\tau\Delta}
		[\eta_\tau u(t+\tau)]\right\}\notag\\
		&=\Pi_{\nu,r}^+e^{-i\tau\Delta}
		[-i\Delta(\eta_\tau u)+(\partial_\tau\eta_\tau)u+
		\eta_\tau\partial_\tau u(t+\tau)]\notag\\
		&=\Pi_{\nu,r}^+e^{-i\tau\Delta}
		[(\partial_\tau\eta_\tau)u(t+\tau)-i[\Delta,\eta_\tau]u(t+\tau)]\notag\\
		&\quad+i\Pi_{\nu,r}^+e^{-i\tau\Delta}
		[\eta_\tau|x|^{-b}|u(t+\tau)|^{\frac{4-2b}{d}}u(t+\tau)].
	\notag
	\end{align}
	Integrating over the finite interval $[S,T]$, we obtain
	\begin{align}
		&\int_S^T\Pi_{\nu,r}^+e^{-i\tau\Delta}
		[\eta_\tau|x|^{-b}|u(t+\tau)|^{\frac{4-2b}{d}}u(t+\tau)]\dd\tau\notag\\
		&=-i\left[\Pi_{\nu,r}^+e^{-i\tau\Delta}
		(\eta_\tau u(t+\tau))\right]_{\tau=S}^{\tau=T}\notag\\
		&\quad+i\int_S^T\Pi_{\nu,r}^+e^{-i\tau\Delta}
		[(\partial_\tau\eta_\tau)u(t+\tau)]\dd\tau\notag\\
		&\quad+\int_S^T\Pi_{\nu,r}^+e^{-i\tau\Delta}
		[[\Delta,\eta_\tau]u(t+\tau)]\dd\tau.
		\label{eq:interior-identity}
	\end{align}
	
	To estimate the commutator term, we first write it in divergence form:
	\begin{equation}\label{eq:comm-div}
		[\Delta,\eta_\tau]u
		=2\nabla\eta_\tau\cdot\nabla u+(\Delta\eta_\tau)u
		=2\diver((\nabla\eta_\tau)u)-(\Delta\eta_\tau)u.
	\end{equation}
	Expressing the resulting terms in kernel form gives
	\begin{align}
		&\Pi_{\nu,r}^+e^{-i\tau\Delta}[[\Delta,\eta_\tau]u](x)\notag\\
		&=-2\int\nabla_yK_{\nu,r}^+(\tau;x,y)
		\cdot\nabla\eta_\tau(y)u(y)\dd y -\int K_{\nu,r}^+(\tau;x,y)(\Delta\eta_\tau)(y)u(y)\dd y.
\notag
	\end{align}
	Moreover, direct differentiation of \eqref{eq:eta} yields
	\begin{equation}\label{eq:eta-derivatives}
		\|\partial_\tau\eta_\tau\|_{L_x^\infty}\lesssim\frac\nu{r+\nu\tau},
		\quad\|\nabla\eta_\tau\|_{L_x^\infty}\lesssim\frac1{r+\nu\tau},
		\quad\|\Delta\eta_\tau\|_{L_x^\infty}\lesssim\frac1{(r+\nu\tau)^2}.
	\end{equation}
	
	\textit{Step 2: endpoint and error estimates.}
	By \eqref{eq:localized-L2} and mass conservation,
	\begin{equation}\label{eq:interior-endpoint}
		\|\Pi_{\nu,r}^+e^{-i\tau\Delta}(\eta_\tau u(t+\tau))\|_{L_x^2}
		\le C_m\|u_0\|_{L_x^2}(\nu r+\nu^2\tau)^{-m}.
	\end{equation}
	Hence, the sum of the two endpoint norms in \eqref{eq:interior-identity} is bounded by $2C_m\|u_0\|_{L_x^2}(\nu r+\nu^2S)^{-m}$. Furthermore, \eqref{eq:localized-L2}, \eqref{eq:localized-div}, \eqref{eq:comm-div}, and \eqref{eq:eta-derivatives} yield
	\begin{align}
		&\|\Pi_{\nu,r}^+e^{-i\tau\Delta}[(\partial_\tau\eta_\tau)u(t+\tau)]\|_{L_x^2}
		\le C_m\|u_0\|_{L_x^2}
		\frac\nu{r+\nu\tau}(\nu r+\nu^2\tau)^{-m},\label{eq:timecut-error}\\
		&\|\Pi_{\nu,r}^+e^{-i\tau\Delta}[[\Delta,\eta_\tau]u(t+\tau)]\|_{L_x^2} \le C_m\|u_0\|_{L_x^2}
		\left(\frac\nu{r+\nu\tau}+\frac1{(r+\nu\tau)^2}\right)
		(\nu r+\nu^2\tau)^{-m}.
		\label{eq:comm-error}
	\end{align}
	
	Both scalar time integrals can be evaluated explicitly. With the substitution $z=\nu r+\nu^2\tau$, we obtain
	\begin{align}
		\int_S^\infty\frac\nu{r+\nu\tau}(\nu r+\nu^2\tau)^{-m}\dd\tau
		&=\frac1m(\nu r+\nu^2S)^{-m},\label{eq:timeintegral-one}\\
		\int_S^\infty\frac1{(r+\nu\tau)^2}(\nu r+\nu^2\tau)^{-m}\dd\tau
		&=\frac1{m+1}(\nu r+\nu^2S)^{-m-1}.
		\label{eq:timeintegral-two}
	\end{align}
	Since $\nu r\ge1$, the second integral is bounded by a constant multiple of the first decay factor. Substituting these estimates into \eqref{eq:interior-identity} establishes \eqref{eq:innerplus}, uniformly in the observation time.
	
	\textit{Step 3: convergence at the future endpoint.}
	For fixed $\nu,r,t$, estimate \eqref{eq:interior-endpoint} implies that
	\[
	\Pi_{\nu,r}^+e^{-iT\Delta}[\eta_Tu(t+T)]\to0
	\quad\text{strongly in }L^2.
	\]
	Alternatively, applying \eqref{eq:innerplus} with $S=T_1$ and $T=T_2$ shows directly that the integrals with finite upper limits form a Cauchy family in $L^2$ as $T_1,T_2\to\infty$. Hence these integrals converge strongly in $L^2$. The future integral over the entire half-line is understood in this sense.
	
	\textit{Step 4: the backward contribution.}
	For $0\le\tau\le t$, the equation gives
	\[
	\partial_\tau u(t-\tau)=-i\Delta u(t-\tau)
	-i|x|^{-b}|u(t-\tau)|^{\frac{4-2b}{d}}u(t-\tau).
	\]
	Consequently,
	\begin{align}
		&\frac{\dd}{\dd\tau}
		\{\Pi_{\nu,r}^-e^{i\tau\Delta}[\eta_\tau u(t-\tau)]\}\notag\\
		&=\Pi_{\nu,r}^-e^{i\tau\Delta}
		[(\partial_\tau\eta_\tau)u(t-\tau)+i[\Delta,\eta_\tau]u(t-\tau)]\notag\\
		&\quad-i\Pi_{\nu,r}^-e^{i\tau\Delta}
		[\eta_\tau|x|^{-b}|u(t-\tau)|^{\frac{4-2b}{d}}u(t-\tau)].
\notag
	\end{align}
	Integrating over $[S,T]$, where $0\le S\le T\le t$, yields
	\begin{align}
		&\int_S^T\Pi_{\nu,r}^-e^{i\tau\Delta}
		[\eta_\tau|x|^{-b}|u(t-\tau)|^{\frac{4-2b}{d}}u(t-\tau)]\dd\tau\notag\\
		&=i\left[\Pi_{\nu,r}^-e^{i\tau\Delta}(\eta_\tau u(t-\tau))\right]_{\tau=S}^{\tau=T}\notag\\
		&\quad-i\int_S^T\Pi_{\nu,r}^-e^{i\tau\Delta}
		[(\partial_\tau\eta_\tau)u(t-\tau)]\dd\tau\notag\\
		&\quad+\int_S^T\Pi_{\nu,r}^-e^{i\tau\Delta}
		[[\Delta,\eta_\tau]u(t-\tau)]\dd\tau.
\notag
	\end{align}
	The corresponding minus-sign kernel estimates yield the analogues of \eqref{eq:interior-endpoint}--\eqref{eq:comm-error}. Integrating the resulting nonnegative upper bounds over $[S,\infty)$ and applying \eqref{eq:timeintegral-one}--\eqref{eq:timeintegral-two}, we obtain \eqref{eq:innerminus}.
\end{proof}

Having controlled the interior contribution, we turn to the exterior nonlinear terms. The decay of the coefficient $|x|^{-b}$ allows us to gain additional frequency decay through a bootstrap argument. More precisely, we first use a provisional frequency-decay bound for $u$ to estimate the frequency-localized nonlinearity, and then apply the dual Strichartz estimates to improve the decay exponent.

For $0\le s<1$, assume
\begin{equation}\label{eq:bootstrap-hyp}
	\text{for every }a>0,\qquad
	\sup_{t\ge0}\|\varphi_{>a}P_\mu u(t)\|_{L_x^2}
	\le C_{u,a,s}\mu^{-s}\quad\text{for all dyadic }\mu\ge1.
\end{equation}
When $s=0$, this bound follows directly from mass conservation. The assumption is imposed for every fixed positive radius to accommodate the smaller cutoff radii used below. We begin with two consequences of \eqref{eq:bootstrap-hyp}, which provide the high-frequency tail and low-frequency gradient estimates needed to control the exterior nonlinearity.

	\begin{lemma}[Exterior high-frequency tails and low-frequency gradients]\label{lem:bootstrap-tools}
		If \eqref{eq:bootstrap-hyp} holds, then for each fixed $a>0$ and $\nu\ge1$,
		\begin{align}
			\|\varphi_{>a}P_{>\nu}u(t)\|_{L_x^2}&\lesssim_{u,a,s}\nu^{-s},\label{eq:localhigh}\\
			\|\varphi_{>a}\nabla P_{\le\nu}u(t)\|_{L_x^2}&\lesssim_{u,a,s}\nu^{1-s}.
			\label{eq:locallow}
		\end{align}
	\end{lemma}
	
	\begin{proof}
		For $s>0$, decompose $P_{>\nu}u=\sum_{\mu>\nu}P_\mu u$ in $L^2$. Applying \eqref{eq:bootstrap-hyp}, we obtain
		\[
		\sum_{\mu>\nu}\|\varphi_{>a}P_\mu u\|_{L_x^2}
		\le C_{u,a,s}\sum_{\mu>\nu}\mu^{-s}
		\le C'_{u,a,s}\nu^{-s}.
		\]
		For $s=0$, we instead use $\|\varphi_{>a}P_{>\nu}u\|_{L_x^2}\lesssim\|u_0\|_{L_x^2}$. This establishes \eqref{eq:localhigh}.
		
		We next prove \eqref{eq:locallow}. For each frequency band $\mu\ge1$, insert $\widetilde P_\mu P_\mu=P_\mu$ and decompose the input in physical space:
		\begin{align*}
			\varphi_{>a}\nabla P_\mu u
			={}&\varphi_{>a}\nabla\widetilde P_\mu
			[\varphi_{>a/2}P_\mu u] +\varphi_{>a}\nabla\widetilde P_\mu
			[\varphi_{\le a/2}P_\mu u].
		\end{align*}
		By Bernstein's inequality and \eqref{eq:bootstrap-hyp}, the first term is bounded by $C_{u,a,s}\mu^{1-s}$. For the second term, the mismatch estimate gives $C_a\mu^{-1}\|P_\mu u\|_{L_x^2}\le C_a\mu^{-1}\|u_0\|_{L_x^2}$. Consequently,
		\[
		\|\varphi_{>a}\nabla P_\mu u(t)\|_{L_x^2}
		\le C_{u,a,s}\mu^{1-s}+C_a\|u_0\|_{L_x^2}\mu^{-1}.
		\]
		Finally, write $P_{\le\nu}=P_{\le1}+\sum_{1<\mu\le\nu}P_\mu$. The low-frequency term is bounded by $C\|u_0\|_{L_x^2}$.
		On the other hand,  since   $s<1$, we have
		\[
		\sum_{1<\mu\le\nu}\mu^{1-s}\lesssim_s\nu^{1-s},
		\qquad\sum_{1<\mu\le\nu}\mu^{-1}\lesssim1.
		\]
		Combining these estimates with the bound for the low-frequency term yields \eqref{eq:locallow}.
	\end{proof}

These estimates allow us to control the frequency-localized nonlinearity away from the origin.

\begin{lemma}[The complete truncated nonlinearity]\label{lem:outer-frequency}
	Suppose that \eqref{eq:bootstrap-hyp} holds. Then, for every fixed $r>0$, every $R\ge r/64$, and every $\nu\ge1$,
	\begin{align}
		&\left\|P_\nu\left[(1-\eta(x/R))|x|^{-b}
		|u(t)|^{\frac{4-2b}{d}}u(t)\right]\right\|_{L_x^{\frac{2d}{d+4-2b}}}
		\le C_{u,r,s}R^{-b}\nu^{-s}.
		\label{eq:outer-frequency}
	\end{align}
	The constant is uniform in the radius $R$.
\end{lemma}

\begin{proof}
	By H\"older's inequality and Lemma~\ref{lem:bootstrap-tools}, we obtain
	\begin{align}
		&\Bigl\|P_\nu\Bigl[(1-\eta(x/R))|x|^{-b}
		\bigl(|u(t)|^{\frac{4-2b}{d}}u(t) 
		-|P_{\le\nu}u(t)|^{\frac{4-2b}{d}}P_{\le\nu}u(t)\bigr)\Bigr]
		\Bigr\|_{L_x^{\frac{2d}{d+4-2b}}}\notag\\
		&\quad\le C R^{-b}\|u_0\|_{L_x^2}^{\frac{4-2b}{d}}
		\|\ind_{|x|\ge R}P_{>\nu}u(t)\|_{L_x^2}\notag\\
		&\quad\le C_{u,r,s}R^{-b}\nu^{-s}.
		\label{eq:outer-difference}
	\end{align}
	
	For the low-frequency contribution, Bernstein's and H\"older's inequalities, together with Lemma~\ref{lem:bootstrap-tools}, yield
	\begin{align}
		&\Bigl\|P_\nu\Bigl[(1-\eta(x/R))|x|^{-b}
		|P_{\le\nu}u(t)|^{\frac{4-2b}{d}}P_{\le\nu}u(t)\Bigr]
		\Bigr\|_{L_x^{\frac{2d}{d+4-2b}}}\notag\\
		&\quad\le C\nu^{-1}R^{-b-1}
		\|u_0\|_{L_x^2}^{1+\frac{4-2b}{d}} +C\nu^{-1}R^{-b}\|u_0\|_{L_x^2}^{\frac{4-2b}{d}}
		\|\ind_{|x|\ge R}\nabla P_{\le\nu}u(t)\|_{L_x^2}\notag\\
		&\quad\le C_{u,r,s}\bigl(R^{-b-1}\nu^{-1}+R^{-b}\nu^{-s}\bigr)\notag\\
		&\quad\le C'_{u,r,s}R^{-b}\nu^{-s}.
		\label{eq:outer-low}
	\end{align}
	Combining \eqref{eq:outer-difference} and \eqref{eq:outer-low} completes the proof of Lemma \ref{lem:outer-frequency}.
\end{proof}

We next integrate the exterior nonlinear estimate in time. The spatial decay of the coefficient, together with the dual Strichartz estimates, yields the positive frequency gain needed for the bootstrap.

	\begin{lemma}[Exterior Duhamel terms]\label{lem:outer-time}
		Assume \eqref{eq:bootstrap-hyp}, and let $\varepsilon$ be as in \eqref{eq:epsilon}. Then
		\begin{equation}\label{eq:outer-plus}
			\begin{split}
				&\Biggl\|\int_0^\infty\Pi_{\nu,r}^+e^{-i\tau\Delta}
				\Biggl[\left(1-\eta\!\left(\frac{64x}{r+\nu\tau}\right)\right)
				|x|^{-b} 
				|u(t+\tau)|^{\frac{4-2b}{d}}u(t+\tau)\Biggr]\dd\tau\Biggr\|_{L_x^2}
				\le C_{u,r,s}\nu^{-s-2\varepsilon}.
			\end{split}
		\end{equation}
 The analogous past integral, with $\Pi_{\nu,r}^-e^{i\tau\Delta}$, $u(t-\tau)$, and $0\le\tau\le t$, satisfies the same estimate uniformly in $t$.
	\end{lemma}

\begin{proof}
	We first consider the integral with a finite upper limit $T$. By the factorization \eqref{eq:Pi-factor}, we have
	\begin{equation} 
		\begin{split}
			&\int_0^T\Pi_{\nu,r}^+e^{-i\tau\Delta}
			\Biggl[\left(1-\eta\!\left(\frac{64x}{r+\nu\tau}\right)\right)
			|x|^{-b}|u(t+\tau)|^{\frac{4-2b}{d}}u(t+\tau)\Biggr]\dd\tau\\
			&\quad=\widetilde\Pi_{\nu,r}^+
			\int_0^T e^{-i\tau\Delta}P_\nu
			\Biggl[\left(1-\eta\!\left(\frac{64x}{r+\nu\tau}\right)\right)
			|x|^{-b} 
			|u(t+\tau)|^{\frac{4-2b}{d}}u(t+\tau)\Biggr]\dd\tau.
		\end{split}\notag
	\end{equation}
	We now apply Lemma~\ref{lem:outer-frequency} with the moving radius $R=(r+\nu\tau)/64$, distinguishing two cases according to the value of $b$.
	
	\textit{Case 1: $1\le b<2$.}
	We choose the following admissible Strichartz pair and its dual:
	\[
	\left(\frac{2}{2-b},\frac{2d}{d-4+2b}\right),
	\qquad
	\left(\frac{2}{b},\frac{2d}{d+4-2b}\right),
	\]
	respectively. Applying the dual Strichartz estimate in Proposition~\ref{prop:strichartz}, together with Lemma~\ref{lem:outer-frequency}, we obtain
	\begin{equation} 
		\begin{split}
			&\Biggl\|\int_0^T\Pi_{\nu,r}^+e^{-i\tau\Delta}
			\Biggl[\left(1-\eta\!\left(\frac{64x}{r+\nu\tau}\right)\right)
			|x|^{-b} 
			|u(t+\tau)|^{\frac{4-2b}{d}}u(t+\tau)\Biggr]\dd\tau\Biggr\|_{L_x^2}\\
			&\quad\le C_{u,r,s}\nu^{-s}
			\|(r+\nu\tau)^{-b}\|_{L_\tau^{2/b}(0,T)}\\
			&\quad\le C_{u,r,s}r^{-b/2}\nu^{-s-b/2}.
		\end{split}\notag
	\end{equation}
	This establishes \eqref{eq:outer-plus} with $\varepsilon=b/4$.
	
	\textit{Case 2: $1/2<b<1$.}
	In this case, the time exponent $2/(2-b)$ is less than $2$, so we instead use the admissible endpoint $(2,2d/(d-2))$, whose dual spatial exponent is $2d/(d+2)$. For the frequency-localized output, Bernstein's inequality gives
	\begin{align}
		&\left\|P_\nu\left[(1-\eta(x/R))|x|^{-b}
		|u(t)|^{\frac{4-2b}{d}}u(t)\right]\right\|_{L_x^{\frac{2d}{d+2}}}\notag\\
		&\quad\le C\nu^{1-b}
		\left\|P_\nu\left[(1-\eta(x/R))|x|^{-b}
		|u(t)|^{\frac{4-2b}{d}}u(t)\right]\right\|_{L_x^{\frac{2d}{d+4-2b}}}\notag\\
		&\quad\le C_{u,r,s}R^{-b}\nu^{1-b-s}.
		\label{eq:outer-bernstein}
	\end{align}
	The corresponding time integral can be evaluated explicitly:
	\begin{equation}\label{eq:weight-time-small}
		\begin{aligned}
			\|(r+\nu\tau)^{-b}\|_{L_\tau^2(0,\infty)}
			&=\left(\int_0^\infty(r+\nu\tau)^{-2b}\dd\tau\right)^{1/2}\\
			&=\left(\frac{r^{1-2b}}{\nu(2b-1)}\right)^{1/2}
			=\frac{r^{1/2-b}}{\sqrt{2b-1}}\nu^{-1/2}.
		\end{aligned}
	\end{equation}
	This is precisely where the assumption $b>1/2$ is needed. Combining \eqref{eq:outer-bernstein} with the endpoint dual Strichartz estimate, we obtain
	\begin{equation}\label{eq:outer-smallb}
		\begin{split}
			&\Biggl\|\int_0^T\Pi_{\nu,r}^+e^{-i\tau\Delta}
			\Biggl[\left(1-\eta\!\left(\frac{64x}{r+\nu\tau}\right)\right)
			|x|^{-b} 
			|u(t+\tau)|^{\frac{4-2b}{d}}u(t+\tau)\Biggr]\dd\tau\Biggr\|_{L_x^2}\\
			&\quad\le C_{u,r,s}\nu^{1-b-s}
			\|(r+\nu\tau)^{-b}\|_{L_\tau^2(0,T)}\\
			&\quad\le C_{u,r,s}\frac{r^{1/2-b}}{\sqrt{2b-1}}
			\nu^{-s-(b-1/2)}.
		\end{split}
	\end{equation}
	This proves \eqref{eq:outer-plus} with $2\varepsilon=b-1/2$.
	
	In both cases, the estimates are uniform in $T$. Applying the same argument over $[T_1,T_2]$, we observe that the scalar weight belongs to $L_\tau^{2/b}$ in Case~1 and to $L_\tau^2$ in Case~2. Consequently, its norm over $[T_1,\infty)$ tends to zero as $T_1\to\infty$. It follows that the integrals with finite upper limits form a Cauchy family in $L_x^2$ and hence converge strongly.
	
	For the past contribution, we use $\Pi^-_{\nu,r}=\widetilde\Pi^-_{\nu,r}P_\nu$ and the sign-reversed dual Strichartz estimate, extending the source by zero outside $[0,t]$. The pointwise spatial bound remains \eqref{eq:outer-frequency}, while the scalar time factors are controlled by  \eqref{eq:weight-time-small}. This yields the same estimate, with a constant independent of $t$.
\end{proof}

	\subsection{Proof of Theorem~\ref{thm:exterior}}
	
	We are now ready to prove the exterior frequency decay estimate. Combining the interior and exterior estimates, we first obtain an improvement of the bootstrap hypothesis. Iterating this improvement yields the desired decay.

\begin{proof}[Proof of Theorem~\ref{thm:exterior}]
	We first consider the case $\nu r\ge1$. Apply the reduced Duhamel formula to $\Pi^+_{\nu,r}u(t)$ and the ordinary Duhamel formula starting at time $0$ to $\Pi^-_{\nu,r}u(t)$. By \eqref{eq:Pi-sum}, we obtain
	\begin{align}
		\varphi_{>r}P_\nu u(t)
		={}&-i\int_0^\infty\Pi^+_{\nu,r}e^{-i\tau\Delta}
		[|x|^{-b}|u(t+\tau)|^{\frac{4-2b}{d}}u(t+\tau)]\dd\tau\notag\\
		&+\Pi^-_{\nu,r}e^{it\Delta}u_0\notag\\
		&+i\int_0^t\Pi^-_{\nu,r}e^{i\tau\Delta}
		[|x|^{-b}|u(t-\tau)|^{\frac{4-2b}{d}}u(t-\tau)]\dd\tau.
		\label{eq:splitDuhamel}
	\end{align}
	We then insert the spatial partition $1=\eta(64x/(r+\nu\tau))+[1-\eta(64x/(r+\nu\tau))]$ into both nonlinear source terms. Assuming \eqref{eq:bootstrap-hyp} and choosing  $m>s+2\varepsilon$, Lemmas~\ref{lem:inner} and \ref{lem:outer-time} show that the first and third terms on the right-hand side of \eqref{eq:splitDuhamel} are bounded in $L_x^2$ by $C_{u,r,s}\nu^{-s-2\varepsilon}$. Consequently,
	\begin{equation}\label{eq:bootstrap-step}
		\sup_{t\ge0}\|\varphi_{>r}P_\nu u(t)\|_{L_x^2}
		\le C_{u,r,s}\nu^{-s-2\varepsilon}+C\|P_\nu u_0\|_{L_x^2}.
	\end{equation}
	
	Since $u_0\in H^1$, we have
	$\|P_\nu u_0\|_{L_x^2}\le C\nu^{-1}\|\nabla u_0\|_{L_x^2}$. Thus, \eqref{eq:bootstrap-step} implies \eqref{eq:bootstrap-hyp} for every exponent $\sigma$ satisfying
	\[
	0\le\sigma<1,\qquad \sigma\le s+2\varepsilon.
	\]
	Here, each bootstrap step may require estimates at a smaller fixed positive radius. Since \eqref{eq:bootstrap-hyp} is assumed to hold for every such radius, this causes no difficulty, and no uniformity as $r\downarrow0$ is required.
	
	Starting with $s=0$, we increase the exponent by at most $2\varepsilon$ at each step. Since $\varepsilon>0$, finitely many iterations suffice to reach $s=1-\varepsilon<1$. Applying \eqref{eq:bootstrap-step} once more with $s=1-\varepsilon$ yields \eqref{eq:exterior} whenever $\nu r\ge1$.
	
	Finally, for each fixed $r>0$, there are only finitely many dyadic frequencies $\nu\ge1$ satisfying $\nu r<1$. For these frequencies, the left-hand side of \eqref{eq:exterior} is bounded by $C\|u_0\|_{L_x^2}$. Enlarging the constant $C_{u,r}$ if necessary, we obtain \eqref{eq:exterior} for all $\nu\ge1$, as desired.
\end{proof}

	\section{Kinetic-energy localization and rigidity}\label{sec:kinetic}
	
	In this section, we establish the rigidity of almost-periodic solutions at the ground-state mass. We first use the exterior frequency decay estimate from Section~\ref{sec:exterior} to localize their kinetic energy. We then combine the resulting localization estimates with a localized virial argument to show that the energy must vanish, leading to the desired classification.
	
	\subsection{Kinetic-energy localization}\label{Kinetic-energy localization}
	
	Almost periodicity provides spatial localization of the mass, but does not by itself ensure localization of the kinetic energy. The exterior frequency decay estimate allows us to obtain the additional control needed for the following result.

	\setcounter{theorem}{1}
	\begin{theorem}[Kinetic-energy localization]\label{thm:kinetic}
		Assume the hypotheses of Theorem~\ref{thm:exterior}. Suppose in addition that $u$ is almost periodic modulo scaling on $[0,\infty)$, with frequency scale $N(t)$. For every $\delta>0$ there exists $C(\delta)>0$ such that
		\begin{equation}\label{eq:kinetic-localization}
			\sup_{t\ge0}
			\|\nabla u(t)\|_{L_x^2(|x|>C(\delta)\sqrt{1+N(t)^{-2}})}<\delta.
		\end{equation}
	\end{theorem}
	
	To prove Theorem~\ref{thm:kinetic}, we use an auxiliary estimate that controls the exterior gradient in terms of the spatial mass tail and a high-frequency remainder. The latter is controlled by Theorem~\ref{thm:exterior} and the $H^1$ regularity of the initial data.
	
	\setcounter{theorem}{0}
	\begin{lemma}[A localized gradient estimate with an initial-data tail]
		\label{lem:kinetic-estimate}
		Under the hypotheses of Theorem~\ref{thm:exterior}, let $\varepsilon>0$ be the exponent fixed in \eqref{eq:epsilon}. For every dyadic $K\ge8$, every $R\ge4$, and every $t\ge0$, we have 
		\begin{align}
			\|\varphi_{>R}\nabla u(t)\|_{L_x^2}
			\le C_u\bigl(&K\|u(t)\|_{L_x^2(|x|>R/2)}+K^{-1}+K^{-\varepsilon} +\|\nabla P_{>K/8}u_0\|_{L_x^2}\bigr).
			\label{eq:kinetic-estimate}
		\end{align}
		The constant is independent of $t,K,R$.
	\end{lemma}
	\setcounter{theorem}{2}
	
	Assuming Lemma~\ref{lem:kinetic-estimate} for the moment, we first complete the proof of Theorem~\ref{thm:kinetic}.
	
	\begin{proof}[Proof of Theorem~\ref{thm:kinetic}]
		Fix $\delta>0$. Since $u_0\in H^1(\R^d)$, we may choose
		a sufficiently large dyadic $K\ge8$, independent of $t$, such that
		\[
		C_u\left(K^{-1}+K^{-\varepsilon}
		+\|\nabla P_{>K/8}u_0\|_{L_x^2}\right)<\frac{\delta}{2}.
		\]
		With $K$ fixed, almost periodicity \eqref{eq:APtails} allows us
		to choose $A\ge4$ sufficiently large so that
		\[
		\sup_{t\ge0}\|u(t)\|_{L_x^2(|x|>A/(2N(t)))}
		<\frac{\delta}{2C_uK}.
		\]
		For each $t\ge0$, set
		\[
		R(t)=A\sqrt{1+N(t)^{-2}}
		\ge\max\left\{4,\frac{A}{N(t)}\right\}.
		\]
		Applying Lemma~\ref{lem:kinetic-estimate} with this choice of
		$R(t)$, we obtain
		\[
		\sup_{t\ge0}
		\|\varphi_{>R(t)}\nabla u(t)\|_{L_x^2}<\delta.
		\]
		This completes the proof of  \eqref{eq:kinetic-localization}. 	 
	\end{proof}

	We now prove Lemma~\ref{lem:kinetic-estimate} by separating the low- and high-frequency contributions. The argument follows the approach in \cite[Section~3]{KLVZ09}; we include the details for completeness.

	\begin{proof}[Proof of Lemma~\ref{lem:kinetic-estimate}]
		We decompose $\varphi_{>R}\nabla u$ into its $P_{\le K}$ and $P_{>K}$ components. For the low-frequency component, Bernstein's inequality and the mismatch estimates in Lemma~\ref{lem:mismatch} yield
		\begin{align*}
			\|P_{\le K}(\varphi_{>R}\nabla u)\|_{L_x^2}
			&\lesssim \|\varphi_{>R}\nabla P_{\le4K}(\varphi_{\le R/2}u)\|_{L_x^2}
			+K\|\varphi_{>R/2}u\|_{L_x^2}\\
			&\quad+\sum_{\mu>4K}
			\|P_{\le K}\varphi_{>R}\nabla P_\mu u\|_{L_x^2}\\
			&\lesssim_u K^{-1}R^{-2}
			+K\|u\|_{L_x^2(|x|>R/2)}.
		\end{align*}
		For the high-frequency component, the Littlewood--Paley square-function estimate gives
		\[
		\|P_{>K}(\varphi_{>R}\nabla u)\|_{L_x^2}^2
		\lesssim
		\sum_{\nu>K}
		\|P_\nu(\varphi_{>R}\nabla u)\|_{L_x^2}^2.
		\]
		For each $\nu>K$, we split the input frequencies into three ranges:
		$\mu<\nu/4$, $\nu/4\le\mu\le4\nu$, and $\mu>4\nu$.
		By the frequency mismatch estimate in Lemma~\ref{lem:mismatch},
		\[
		\begin{aligned}
			&\sum_{\mu<\nu/4}
			\|P_\nu\varphi_{>R}\nabla P_\mu u\|_{L_x^2}
			+\sum_{\mu>4\nu}
			\|P_\nu\varphi_{>R}\nabla P_\mu u\|_{L_x^2}\\
			&\qquad\lesssim_u
			R^{-2}
			\left(
			\nu^{-2}\sum_{\mu<\nu/4}\mu
			+\sum_{\mu>4\nu}\mu^{-1}
			\right)
			\lesssim_u \nu^{-1}R^{-2}.
		\end{aligned}
		\]
		For comparable frequencies, we insert
		$\widetilde P_\mu P_\mu=P_\mu$ and apply the spatial mismatch estimate together with Bernstein's inequality to obtain
		\[
		\begin{aligned}
			\|\varphi_{>R}\nabla P_\mu u\|_{L_x^2}
			&\le
			\|\varphi_{>R}\nabla\widetilde P_\mu
			(\varphi_{\le R/2}P_\mu u)\|_{L_x^2}\\
			&\quad+
			\|\varphi_{>R}\nabla\widetilde P_\mu
			(\varphi_{>R/2}P_\mu u)\|_{L_x^2}\\
			&\lesssim_u
			\mu^{-1}R^{-2}
			+\mu\|\varphi_{>R/2}P_\mu u\|_{L_x^2}.
		\end{aligned}
		\]
		Consequently, summing over the output frequencies and applying Theorem~\ref{thm:exterior}, we obtain
		\[
		\begin{aligned}
			\|P_{>K}(\varphi_{>R}\nabla u)\|_{L_x^2}^2
			&\quad\lesssim_u
			R^{-4}\sum_{\nu>K}\nu^{-2}
			+\sum_{\nu>K}
			\sum_{\nu/4\le\mu\le4\nu}
			\mu^2\|\varphi_{>R/2}P_\mu u\|_{L_x^2}^2\\
			&\quad\lesssim_u
			K^{-2}R^{-4}
			+\sum_{\mu>K/4}
			\mu^2\|\varphi_{>R/2}P_\mu u\|_{L_x^2}^2.\\
			&\quad\lesssim_u K^{-2}R^{-4}+K^{-2\varepsilon}
			+\sum_{\mu>K/4}\mu^2\|P_\mu u_0\|_{L_x^2}^2\\
			&\quad\lesssim_u K^{-2}R^{-4}+K^{-2\varepsilon}
			+\|\nabla P_{>K/8}u_0\|_{L_x^2}^2,
		\end{aligned}
		\]
		which, upon taking square roots, completes the proof of \eqref{eq:kinetic-estimate}.
	\end{proof}

	The localized virial argument also requires control of the nonlinear potential-energy tail. The following estimate, based on the sharp Gagliardo--Nirenberg inequality, provides this control in terms of localized mass and kinetic energy.

	\begin{lemma}[Localization of potential energy]\label{lem:potential-tail}
		For any $R>0$ and $u(t)\in H^1(\R^d)$, we have
		\begin{equation} 
			\begin{aligned}
				&\int_{|x|>R}|x|^{-b}|u(t,x)|^{2+\frac{4-2b}{d}}\dd x\lesssim
				\|u(t)\|_{L_x^2(|x|>R/2)}^{\frac{4-2b}{d}}
				\bigl(\|\nabla u(t)\|_{L_x^2(|x|>R/2)}^2+R^{-2}M(u)\bigr).
			\end{aligned}\notag
		\end{equation}
	\end{lemma}
	
	\begin{proof}
		The  estimate follows directly from the sharp
		Gagliardo--Nirenberg inequality \eqref{eq:sharp}:
		\begin{align*}
			&\int_{|x|>R}|x|^{-b}|u(t,x)|^{2+\frac{4-2b}{d}}\dd x\\
			&\quad\le \int_{\R^d}|x|^{-b}|\varphi_{>R/2}u(t)|^{2+\frac{4-2b}{d}}\dd x\\
			&\quad\lesssim
			\|\varphi_{>R/2}u(t)\|_{L_x^2}^{\frac{4-2b}{d}}
			\|\nabla(\varphi_{>R/2}u(t))\|_{L_x^2}^2\\
			&\quad\lesssim
			\|u(t)\|_{L_x^2(|x|>R/2)}^{\frac{4-2b}{d}}
			\bigl(\|\nabla u(t)\|_{L_x^2(|x|>R/2)}^2+R^{-2}M(u)\bigr).
		\end{align*}
		
	\end{proof}

	\subsection{Rigidity via the localized virial identity}\label{sec:virial}
	
	We now combine the localization estimates from the preceding subsection with a localized virial argument. The key observation is that positive energy would force the second derivative of a suitable virial functional to remain positive, while its first derivative admits a bound that is incompatible with such growth. We begin by recalling the localized virial identities.
	
	Choose a real smooth radial function $a\in C_c^\infty(\R^d)$ such that $a(x)=|x|^2$ for $|x|\le1$ and $a(x)=0$ for $|x|\ge2$. For $R>0$, set
	\begin{equation}\label{eq:virial-weight}
		a_R(x)=R^2a(x/R),\qquad
		V_R(t)=\int a_R(x)|u(t,x)|^2\dd x.
	\end{equation}
	The following lemma records the localized virial identities and the error arising from spatial truncation.
	
	\begin{lemma}[Localized virial identities, {\cite[Corollary~3.2]{Di18}}]
		\label{lem:virial} Let  $d\ge3$ and $ \frac12<b<\min\left\{2,\frac d2\right\}$.  
		For an $H^1$ strong solution of \eqref{eq:inls}, $V_R$ is twice continuously differentiable on its lifespan and
		\begin{equation}\label{eq:virial-first}
			V_R'(t)=2\int\nabla a_R\cdot\Ima(\overline u\nabla u)\dd x,
		\end{equation}
		\begin{align}
			V_R''(t)
			={}&4\sum_{j,k=1}^d\int\partial_{jk}a_R\,
			\Rea(\partial_j\overline u\,\partial_ku)\dd x
			-\int\Delta^2a_R|u|^2\dd x\notag\\
			&-\frac{2(2-b)}{d+2-b}
			\int\Delta a_R\,|x|^{-b}|u|^{2+\frac{4-2b}{d}}\dd x\notag\\
			&-\frac{2db}{d+2-b}
			\int\frac{x\cdot\nabla a_R}{|x|^{b+2}}
			|u|^{2+\frac{4-2b}{d}}\dd x.
			\label{eq:virial-second}
		\end{align}
		Moreover, substituting \eqref{eq:virial-weight} into \eqref{eq:virial-second},
		one obtains 
		\begin{align}
			|V_R''(t)-16E(u)|\le C\bigg[&R^{-2}M(u)
			+\int_{|x|>R}|\nabla u(t,x)|^2\dd x +\int_{|x|>R}|x|^{-b}|u(t,x)|^{2+\frac{4-2b}{d}}\dd x\bigg].
			\label{eq:virial-error}
		\end{align}
		The constant is independent of $t,R$.
	\end{lemma}

	To control the first derivative of the localized virial functional, we also use the following phase--energy inequality at the ground-state mass.
	
	\begin{proposition}[Phase--energy inequality, {\cite{B04,CG16}}]\label{prop:banica}
		If $f\in H^1$, $M(f)=M(Q)$, and $a\in C_c^\infty(\R^d;\R)$, then $E(f)\ge0$ and
		\begin{equation} 
			\left|\int\nabla a\cdot\Ima(\overline f\nabla f)\dd x\right|
			\le\sqrt{2E(f)}
			\left(\int|\nabla a|^2|f|^2\dd x\right)^{1/2}.\notag
		\end{equation}
	\end{proposition}

	We are now ready to prove the rigidity theorem. The preceding localization estimates make the truncation error in the virial identity small, while the phase--energy inequality controls the first derivative. Together, these bounds rule out positive energy for an almost-periodic threshold solution.

	\begin{theorem}[Rigidity of almost-periodic threshold solutions]
		\label{thm:AP-rigidity}
		Let $d\ge3$ and $\frac12<b<\min\{2,d/2\}$. Suppose that
		$u\in C([0,\infty);H^1(\R^d))$ solves \eqref{eq:inls},
		satisfies $M(u)=M(Q)$, and is almost periodic modulo scaling.
		Then $E(u)=0$, and $u$ has the form \eqref{eq:soliton}.
	\end{theorem}
	
	\begin{proof}
		By the sharp Gagliardo--Nirenberg inequality \eqref{eq:sharp},
		we have $E(u)\ge0$. Suppose, for contradiction, that $E(u)>0$.
		We shall derive a contradiction using the localized virial
		estimate in Lemma~\ref{lem:virial}. This will force $E(u)=0$,
		and the conclusion will then follow from
		Proposition~\ref{prop:classification}.
		
		First, for every $\delta\in(0,1)$, mass localization
		\eqref{eq:APtails} and Theorem~\ref{thm:kinetic} yield
		a constant $A_\delta\ge1$ such that  $\text{whenever}\quad
		R\ge2A_\delta\sqrt{1+N(t)^{-2}},\qquad t\ge0$, we have  
		\begin{align}
			&\|u(t)\|_{L_x^2(|x|>R/2)}
			+\|\nabla u(t)\|_{L_x^2(|x|>R/2)}
			\le\delta. \notag 
		\end{align}
		Moreover, Lemma~\ref{lem:potential-tail} gives
		\begin{equation} 
			\begin{aligned}
				&\int_{|x|>R}|x|^{-b}
				|u(t,x)|^{2+\frac{4-2b}{d}}\dd x \le
				C\delta^{\frac{4-2b}{d}}
				\bigl(\delta^2+R^{-2}M(Q)\bigr).\notag
			\end{aligned}
		\end{equation}
		Consequently, by \eqref{eq:virial-error},
		\begin{align}
			|V_R''(t)-16E(u)|
			\le{}&
			C\left(\delta^2+\delta^{2+\frac{4-2b}{d}}\right)
			+C\left(1+\delta^{\frac{4-2b}{d}}\right)
			R^{-2}M(Q).
			\label{eq:virial-error-small}
		\end{align}

		We now choose $\delta>0$ sufficiently small and then
		fix $R>0$ sufficiently large so that
		\[
		R\ge2A_\delta C_u,
		\qquad
		C\left(\delta^2+\delta^{2+\frac{4-2b}{d}}\right)
		+C\left(1+\delta^{\frac{4-2b}{d}}\right)
		R^{-2}M(Q)\le8E(u).
		\]
		For each $T\ge1$, set
		\begin{equation} 
			R_T=R\sqrt{1+T}.\notag
		\end{equation}
		By \eqref{eq:virial-scale-control}, we have
		\[
		R_T\ge2A_\delta\sqrt{1+N(t)^{-2}},
		\qquad 0\le t\le T.
		\]
		Therefore, \eqref{eq:virial-error-small} yields
		\begin{equation}\label{eq:positive-virial}
			V_{R_T}''(t)\ge8E(u)>0,
			\qquad 0\le t\le T.
		\end{equation}
		
		Next, we estimate the first derivative of the localized
		virial functional. By Proposition~\ref{prop:banica} and
		\eqref{eq:virial-first}, we obtain
		\begin{align}
			|V_R'(t)|
			&\le2\sqrt{2E(u)}
			\left(\int|\nabla a_R|^2|u(t,x)|^2\dd x\right)^{1/2}
			\notag\\
			&\le2\sqrt{2}\,\|\nabla a\|_{L_x^\infty}
			R\sqrt{E(u)M(Q)}.
			\label{eq:virial-first-bound}
		\end{align}
		Integrating \eqref{eq:positive-virial} over $[0,T]$
		and applying \eqref{eq:virial-first-bound} at both
		endpoints, we obtain
		\begin{align*}
			8E(u)T
			&\le V_{R_T}'(T)-V_{R_T}'(0)\\
			&\le4\sqrt{2}\,\|\nabla a\|_{L_x^\infty}
			R\sqrt{1+T}\sqrt{E(u)M(Q)}.
		\end{align*}
	Since $R$ is fixed independently of $T$, the preceding inequality is impossible as $T\to\infty$. Therefore $E(u)=0$, and Proposition~\ref{prop:classification} yields the claimed classification.
	
	\end{proof}
 
\section{Proof of the main results}\label{sec:conclusion}

In this section, we complete the proofs of the main results stated in the introduction. In the first subsection, we combine the compactness result established in Section~\ref{sec:compactness} with the rigidity theorem proved in Section~\ref{sec:kinetic} to classify forward-global nonscattering solutions at the ground-state mass. Together with the finite-time blowup classification of Combet and Genoud~\cite{CG16}, this yields the nonradial scattering result and the threshold trichotomy for $b>1/2$. In the second subsection, we use the pseudoconformal transformation and the finite-time classification to establish the corresponding rigidity and scattering results for finite-variance data throughout the full parameter range \eqref{eq:full-range}.

 \subsection{Rigidity, nonradial scattering, and the threshold trichotomy}
 
 The main theorem follows directly from the compactness result in Proposition~\ref{prop:compactness} and the rigidity theorem for almost-periodic solutions.
 
 \begin{proof}[Proof of Theorem~\ref{thm:main}]
 	By Proposition~\ref{prop:local}, failure of forward $L^2$ scattering means that the forward critical mixed norm is infinite. Proposition~\ref{prop:compactness} gives almost periodicity modulo scaling of the original solution, and Theorem~\ref{thm:AP-rigidity} then yields \eqref{eq:soliton}.
 \end{proof}
 
 We next consider nonradial initial data at the ground-state mass. The finite-time blowup classification and the rigidity result exclude the possible obstructions to global existence and scattering.
 
 \begin{proof}[Proof of Corollary~\ref{cor:nonradial}]
 	Let $(T_-,T_+)$ be the maximal lifespan. If $T_+<\infty$, the $H^1$ blowup alternative and Proposition~\ref{prop:classification} force the solution to take the form \eqref{eq:pseudoconformal}. Its initial datum is then radial about the origin, a contradiction. The same argument applied to the time-reversed solution $\overline{u(-t,x)}$ excludes $T_->-\infty$. Thus the solution exists on all of $\R$.
 	
 	If the solution failed to scatter forward in $L^2$, Theorem~\ref{thm:main} would force its initial datum to be a phase rotation and scaling of the radial ground state, again a contradiction. Time reversal gives backward $L^2$ scattering.  
 	By Proposition~\ref{prop:local}(f), both scattering statements can be upgraded to $H^1$.
 \end{proof}
 
 We now prove the threshold trichotomy by combining the global rigidity result with the finite-time classification. The only additional point is to verify scattering at the infinite endpoint of a pseudoconformal solution.
 
 \begin{proof}[Proof of Corollary~\ref{cor:trichotomy}]
 	Let $(T_-,T_+)$ be the maximal lifespan of $u$.
 	
 	First, suppose that $T_-=-\infty$ and $T_+=\infty$.
 	By Theorem~\ref{thm:main} and its time-reversed version,
 	if $u$ fails to scatter in either time direction, then
 	$u$ must be a standing wave, giving case (ii).
 	Otherwise, the critical spacetime norms are finite
 	in both directions, and Proposition~\ref{prop:local}
 	yields $H^1$ scattering as $t\to\pm\infty$.
 	This gives case (i).
 	
 	Next, suppose that one endpoint of the lifespan is finite.
 	Without loss of generality, assume $T_+=T<\infty$.
 	By the finite-time minimal-mass blowup classification
 	of Combet--Genoud, stated in
 	Proposition~\ref{prop:classification},
 	there exist $\theta\in\R$ and $\lambda>0$ such that
 	\[
 	u(t,x)=e^{i\theta}e^{i\lambda^2/(T-t)}
 	e^{-i|x|^2/[4(T-t)]}
 	\left(\frac{\lambda}{T-t}\right)^{d/2}
 	Q\!\left(\frac{\lambda x}{T-t}\right),
 	\qquad t<T.
 	\]
 	Since this explicit solution is defined on the entire
 	half-line $(-\infty,T)$, it follows that $T_-=-\infty$.
 	In particular, both lifespan endpoints cannot be finite.
 	
 	It remains to show that this pseudoconformal solution
 	scatters in $H^1$ as $t\to-\infty$. 
 	Since  $Q\in H^1(\mathbb{R} ^d)$ and   $\frac{2(d+4-2b)}{d+2-2b}\in (2,\frac{2d}{d-2})$, we deduce from  
 	Sobolev  embedding   that  $	Q\in L^{\frac{2(d+4-2b)}{d+2-2b},2}(\R^d)$. Hence, for every $t_0<T$,
 	\begin{align*}
 		&\int_{-\infty}^{t_0}
 		\|u(t)\|_{L_x^{\frac{2(d+4-2b)}{d+2-2b},2}}^{\frac{2(d+4-2b)}d}
 		\dd t =\int_{-\infty }^{t_0}	\left(\frac{\lambda}{T-t}\right)^2
 		\|Q\|_{L_x^{\frac{2(d+4-2b)}{d+2-2b},2}}^{\frac{2(d+4-2b)}d}dt=
 		\frac{\lambda^2}{T-t_0}
 		\|Q\|_{L_x^{\frac{2(d+4-2b)}{d+2-2b},2}}^{\frac{2(d+4-2b)}d}
 		<\infty.
 	\end{align*}
 	Thus the critical spacetime norm is finite on
 	$(-\infty,t_0]$.  Proposition~\ref{prop:local} then 
 	implies that $u$ scatters in $H^1$ as $t\to-\infty$.
 	This establishes case (iii) when the right lifespan
 	endpoint is finite. The case of a finite left endpoint
 	follows by time reversal.
 \end{proof}
 
 \subsection{The finite-variance case and pseudoconformal compactification}
 \label{subsec:finite-variance}
 
 We now turn to solutions with finite variance, working in the full parameter range \eqref{eq:full-range}. Unlike the preceding arguments, the proof does not rely on exterior frequency decay. Instead, we use the pseudoconformal invariance of \eqref{eq:inls} and the finite-time blowup classification of Combet--Genoud~\cite[Theorem 1 and Lemma 9]{CG16}.
 
 To apply the pseudoconformal transformation in the energy space, we first recall that finite variance is preserved on compact lifespan intervals.
 
 \begin{lemma}[Persistence of finite variance, {\cite[Lemma~3.3]{Di18}}]\label{lem:Sigma-persistence}
 	Assume \eqref{eq:full-range}. Let $u\in C(I;H^1(\R^d))$ be a strong solution of \eqref{eq:inls}, where $I$ is an interval, and suppose $u(t_0)\in\Sigma$ for some $t_0\in I$. Then $u\in C(I;\Sigma)$.  
 \end{lemma}
 
 The next lemma uses this persistence property to compactify the forward time interval. It also relates the critical spacetime norms before and after the transformation, allowing us to characterize scattering through the behavior of the transformed solution at the endpoint.
 
 \begin{lemma}[Pseudo-conformal transformation, {\cite[Lemma~9]{CG16}}]
 	\label{lem:pc-compactification}
 	Assume \eqref{eq:full-range}, and let $u\in C([0,\infty);H^1)$ be a strong solution with $u(0)\in\Sigma$. For $0\le s<1$ set
 	\begin{equation}\label{eq:pc-compactification}
 		v(s,x)=(1-s)^{-d/2}e^{-i|x|^2/[4(1-s)]}
 		u\!\left(\frac{s}{1-s},\frac{x}{1-s}\right).
 	\end{equation}
 	Then $v\in C([0,1);H^1)$ is a strong solution of the same equation,
 	\begin{equation} 
 		v(0,x)=e^{-i|x|^2/4}u_0(x),\qquad M(v)=M(u_0),\notag
 	\end{equation}
 	and, for every $0<a<1$,
 	\begin{align}
 		&\int_0^{a/(1-a)}
 		\|u(t)\|_{L_x^{\frac{2(d+4-2b)}{d+2-2b},2}}^{\frac{2(d+4-2b)}d}\dd t
 		=
 		\int_0^a
 		\|v(s)\|_{L_x^{\frac{2(d+4-2b)}{d+2-2b},2}}^{\frac{2(d+4-2b)}d}\dd s.
 	\notag
 	\end{align}
 	In particular, if $v$ extends as an $H^1$ solution to an open interval containing $[0,1]$, then $u$ scatters forward in $H^1$.
 \end{lemma}
 
 We can now prove the finite-variance classification. The transformed solution either extends beyond the compactified endpoint, in which case the preceding lemma gives scattering, or has a finite-time blowup there, in which case the Combet--Genoud classification applies.
 
 \begin{proof}[Proof of Proposition~\ref{prop:finite-variance}]
 	Apply Lemma~\ref{lem:pc-compactification} to the given forward-global solution and let $v$ be defined by \eqref{eq:pc-compactification}. Its positive maximal $H^1$ lifespan endpoint is at least $1$, since the explicit transform is an $H^1$ strong solution on every compact subinterval of $[0,1)$.
 	
 	If the endpoint is strictly larger than $1$, the last assertion of Lemma~\ref{lem:pc-compactification} gives forward $H^1$ scattering of $u$. If the endpoint is exactly $1$, the $H^1$ blowup alternative and mass conservation imply that $v$ is a finite-time blowup solution of mass $M(Q)$. Proposition~\ref{prop:classification}  then gives constants $\gamma\in\R$ and $\lambda>0$ such that
 	\begin{align}
 		v(s,x)
 		={}&e^{i\gamma}e^{i\lambda^2/(1-s)}
 		e^{-i|x|^2/[4(1-s)]}
 		\left(\frac{\lambda}{1-s}\right)^{d/2}
 		Q\!\left(\frac{\lambda x}{1-s}\right),
 		\qquad 0\le s<1.
 		\label{eq:pc-classified-v}
 	\end{align}
 	Substituting  \eqref{eq:pc-compactification} into (\ref{eq:pc-classified-v})   gives
 	\[
 	u(t,y)=e^{i\gamma}e^{i\lambda^2(1+t)}\lambda^{d/2}Q(\lambda y).
 	\]
 	After absorbing the constant $\lambda^2$ into the phase, this is \eqref{eq:soliton}. Since a standing wave  has infinite critical norm on $[0,\infty)$; it cannot belong to the scattering alternative. This proves the dichotomy and, in particular, the asserted classification of every nonscattering forward-global solution in $\Sigma$.
 \end{proof}
 
 Finally, the finite-variance classification can be combined with the nonradial scattering result from the first subsection to cover the parameter ranges stated in Corollary~\ref{cor:combined-scattering}.
 
 \begin{proof}[Proof of Corollary~\ref{cor:combined-scattering}]
 	Let $(T_-,T_+)$ be the maximal lifespan. In the full range \eqref{eq:full-range}, the finite-time classification in Proposition~\ref{prop:classification} forces any critical-mass solution with a finite endpoint to be radial about the origin at every time. Since $u_0$ is nonradial, both finite endpoints are excluded. Thus $u$ is global in both directions.
 	
 	If $b>1/2$, Corollary~\ref{cor:nonradial} proves the desired two-sided $H^1$ scattering. In the remaining case $0<b\le1/2$, the hypothesis of the present corollary requires $xu_0\in L^2$. Apply Proposition~\ref{prop:finite-variance} to $u$ and then to $\overline{u(-t,x)}$. In each direction, the standing-wave alternative contradicts the nonradiality of $u_0$. Both scattering conclusions follow. The same argument proves two-sided scattering for nonradial data in $\Sigma$ throughout \eqref{eq:full-range}, independently of the exterior frequency estimate.
 \end{proof}
 
  \bigskip
 \noindent\textbf{Acknowledgements:}  X.~Liu was supported by NSFC Grant 12501314. 
 
  \bigskip
\noindent\textbf{Declaration of AI use.}
ChatGPT was used for research-related discussions
and to assist with language editing, manuscript
organization, LaTeX preparation, bibliographic
verification, and consistency checks.
The author critically reviewed and revised
all content and takes full responsibility for the manuscript.

	\bigskip
	
	\noindent\textsc{School of Mathematics, Hangzhou Normal University,
		Hangzhou 311121, China}
	
	\noindent\textit{Email address:} \texttt{liuxuan95@hznu.edu.cn}

\bigskip


\begin{thebibliography}{99}
	
	\bibitem{AT21}
	L. Aloui and S. Tayachi:
	Local well-posedness for the inhomogeneous nonlinear Schr\"odinger equation.
	Discrete Contin. Dyn. Syst. \textbf{41}(11), 5409--5437 (2021).
	
	\bibitem{B04}
	V. Banica:
	Remarks on the blow-up for the Schr\"odinger equation with critical mass on a plane domain.
	Ann. Sc. Norm. Super. Pisa Cl. Sci. (5) \textbf{3}(1), 139--170 (2004).
	
	\bibitem{CC22}
	L. Campos and M. Cardoso:
	On the critical norm concentration for the inhomogeneous nonlinear Schr\"odinger equation.
	J. Dynam. Differential Equations \textbf{34}(3), 2347--2369 (2022).
	
	\bibitem{CF22}
	M. Cardoso and L. G. Farah:
	Blow-up of non-radial solutions for the $L^2$ critical inhomogeneous NLS equation.
	Nonlinearity \textbf{35}(8), 4426--4436 (2022).
	
	\bibitem{CG16}
	V. Combet and F. Genoud:
	Classification of minimal mass blow-up solutions for an $L^2$ critical inhomogeneous NLS.
	J. Evol. Equ. \textbf{16}(2), 483--500 (2016).
	
	\bibitem{Di18}
	V. D. Dinh:
	Blowup of $H^1$ solutions for a class of the focusing inhomogeneous nonlinear Schr\"odinger equation.
	Nonlinear Anal. \textbf{174}, 169--188 (2018).
	
	\bibitem{D12}
	B. Dodson:
	Global well-posedness and scattering for the defocusing, $L^2$-critical nonlinear Schr\"odinger equation when $d\ge3$.
	J. Amer. Math. Soc. \textbf{25}(2), 429--463 (2012).
	
	\bibitem{D15}
	B. Dodson:
	Global well-posedness and scattering for the mass critical nonlinear Schr\"odinger equation with mass below the mass of the ground state.
	Adv. Math. \textbf{285}, 1589--1618 (2015).
	
	\bibitem{D16a}
	B. Dodson:
	Global well-posedness and scattering for the defocusing, $L^2$-critical nonlinear Schr\"odinger equation when $d=2$.
	Duke Math. J. \textbf{165}(18), 3435--3516 (2016).
	
	\bibitem{D16b}
	B. Dodson:
	Global well-posedness and scattering for the defocusing, $L^2$-critical nonlinear Schr\"odinger equation when $d=1$.
	Amer. J. Math. \textbf{138}(2), 531--569 (2016).
	
	\bibitem{D23}
	B. Dodson:
	A determination of the blowup solutions to the focusing NLS with mass equal to the mass of the soliton.
	Ann. PDE \textbf{9}(1), Article 3 (2023).
	
	\bibitem{D24}
	B. Dodson:
	A determination of the blowup solutions to the focusing, quintic NLS with mass equal to the mass of the soliton.
	Anal. PDE \textbf{17}(5), 1693--1760 (2024).
	
	\bibitem{G08}
	F. Genoud:
	Th\'eorie de bifurcation et de stabilit\'e pour une \'equation de Schr\"odinger avec une non-lin\'earit\'e compacte.
	Ph.D. thesis, No. 4233, EPFL, Lausanne (2008) (in French).
	
	\bibitem{G10}
	F. Genoud:
	Bifurcation and stability of travelling waves in self-focusing planar waveguides.
	Adv. Nonlinear Stud. \textbf{10}(2), 357--400 (2010).
	
	\bibitem{G11}
	F. Genoud:
	A uniqueness result for $\Delta u-\lambda u+V(|x|)u^p=0$ on $\mathbb{R}^2$.
	Adv. Nonlinear Stud. \textbf{11}(3), 483--491 (2011).
	
	\bibitem{G12}
	F. Genoud:
	An inhomogeneous, $L^2$-critical, nonlinear Schr\"odinger equation.
	Z. Anal. Anwend. \textbf{31}(3), 283--290 (2012).
	
	\bibitem{GS08}
	F. Genoud and C. A. Stuart:
	Schr\"odinger equations with a spatially decaying nonlinearity: existence and stability of standing waves.
	Discrete Contin. Dyn. Syst. \textbf{21}(1), 137--186 (2008).
	
	\bibitem{Gill00}
	T. S. Gill:
	Optical guiding of laser beam in nonuniform plasma.
	Pramana \textbf{55}(5--6), 835--842 (2000).
	
	
	\bibitem{HK05}
	T. Hmidi and S. Keraani:
	Blowup theory for the critical nonlinear Schr\"odinger equations revisited.
	Int. Math. Res. Not. \textbf{2005}(46), 2815--2828 (2005).
	
	
	\bibitem{KT98}
	M. Keel and T. Tao:
	Endpoint Strichartz estimates.
	Amer. J. Math. \textbf{120}(5), 955--980 (1998).
	
	\bibitem{KM06}
	C. E. Kenig and F. Merle:
	Global well-posedness, scattering and blow-up for the energy-critical, focusing, non-linear Schr\"odinger equation in the radial case.
	Invent. Math. \textbf{166}(3), 645--675 (2006).
	
	\bibitem{KLVZ09}
	R. Killip, D. Li, M. Visan and X. Zhang:
	Characterization of minimal-mass blowup solutions to the focusing mass-critical NLS.
	SIAM J. Math. Anal. \textbf{41}(1), 219--236 (2009).
	
	\bibitem{KTV09}
	R. Killip, T. Tao and M. Visan:
	The cubic nonlinear Schr\"odinger equation in two dimensions with radial data.
	J. Eur. Math. Soc. (JEMS) \textbf{11}(6), 1203--1258 (2009).
	
	\bibitem{KVZ08}
	R. Killip, M. Visan and X. Zhang:
	The mass-critical nonlinear Schr\"odinger equation with radial data in dimensions three and higher.
	Anal. PDE \textbf{1}(2), 229--266 (2008).
	
	\bibitem{KV13}
	R. Killip and M. Visan:
	Nonlinear Schr\"odinger equations at critical regularity.
	In: Evolution Equations, Clay Math. Proc., vol. 17, pp. 325--437.
	American Mathematical Society, Providence (2013).
	
	\bibitem{LZ12}
	D. Li and X. Zhang:
	On the rigidity of solitary waves for the focusing mass-critical NLS in dimensions $d\ge2$.
	Sci. China Math. \textbf{55}(2), 385--434 (2012).
	
	\bibitem{LT94}
	C. S. Liu and V. K. Tripathi:
	Laser guiding in an axially nonuniform plasma channel.
	Phys. Plasmas \textbf{1}(9), 3100--3103 (1994).
	
	\bibitem{LMZ25}
	X. Liu, C. Miao and J. Zheng:
	Global well-posedness and scattering for mass-critical inhomogeneous NLS when $d\ge3$.
	Math. Z. \textbf{311}, Article 54 (2025).
	
	\bibitem{M93}
	F. Merle:
	Determination of blow-up solutions with minimal mass for nonlinear Schr\"odinger equations with critical power.
	Duke Math. J. \textbf{69}(2), 427--454 (1993).
	
	\bibitem{M96}
	F. Merle:
	Nonexistence of minimal blow-up solutions of equations
	$iu_t=-\Delta u-k(x)|u|^{4/N}u$ in $\mathbb{R}^N$.
	Ann. Inst. H. Poincar\'e Phys. Th\'eor. \textbf{64}(1), 33--85 (1996).
	
	\bibitem{ON63}
	R. O'Neil:
	Convolution operators and $L(p,q)$ spaces.
	Duke Math. J. \textbf{30}(1), 129--142 (1963).
	
	\bibitem{TVZ07}
	T. Tao, M. Visan and X. Zhang:
	Global well-posedness and scattering for the defocusing mass-critical nonlinear Schr\"odinger equation for radial data in high dimensions.
	Duke Math. J. \textbf{140}(1), 165--202 (2007).
	
	\bibitem{TVZ08}
	T. Tao, M. Visan and X. Zhang:
	Minimal-mass blowup solutions of the mass-critical NLS.
	Forum Math. \textbf{20}(5), 881--919 (2008).
	
	\bibitem{T84}
	J. F. Toland:
	Uniqueness of positive solutions of some semilinear Sturm--Liouville problems on the half line.
	Proc. Roy. Soc. Edinburgh Sect. A \textbf{97}, 259--263 (1984).
	
	\bibitem{Y91}
	E. Yanagida:
	Uniqueness of positive radial solutions of
	$\Delta u+g(r)u+h(r)u^p=0$ in $\mathbb{R}^n$.
	Arch. Ration. Mech. Anal. \textbf{115}(3), 257--274 (1991).
	
\end{thebibliography}
\end{document}